\documentclass[a4paper, 12pt, twoside, notitlepage]{amsart}

\usepackage{amsmath,amscd}
\usepackage{amssymb}
\usepackage{amsthm}
\usepackage{comment}
\usepackage{graphicx, xcolor}
\usepackage{mathrsfs}
\usepackage[ocgcolorlinks, linkcolor=blue]{hyperref}

\usepackage{bm}
\usepackage{bbm}
\usepackage{url}

\usepackage[utf8]{inputenc}
\usepackage{mathtools,amssymb}
\usepackage{esint}
\usepackage{tikz}
\usepackage{dsfont}
\usepackage{relsize}
\usepackage{url}
\usepackage[shortlabels]{enumitem}
\usepackage{lineno}
\usepackage{enumitem}

 \usepackage{fullpage} 
\usepackage{verbatim}
\usepackage{dsfont}
\numberwithin{equation}{section}

\allowdisplaybreaks

\graphicspath{{images/}}

\newtheorem{theorem}{Theorem}[section]
\newtheorem{lemma}[theorem]{Lemma}

\newtheorem{proposition}[theorem]{Proposition}

\title[Normal operators for momentum ray transforms]{Normal operators for momentum ray transforms,\\  I: The inversion formula}
\author[S.R. Jathar]{Shubham R. Jathar}
\address{ Indian Institute of Science Education and Research (IISER) Bhopal, India}
\email {shubham18@iiserb.ac.in}

\author[M. Kar]{Manas Kar}
\address{ Indian Institute of Science Education and Research (IISER) Bhopal, India}
\email{manas@iiserb.ac.in}

\author[V. P. Krishnan]{Venkateswaran P. Krishnan}
\address{
Centre for Applicable Mathematics, Tata Institute of Fundamental Research, India }
\email{vkrishnan@tifrbng.res.in}

\author[V. A. Sharafutdinov]{Vladimir A. Sharafutdinov}
\address{Sobolev Institute of Mathematics, 4 Koptyug Av., 630090, Novosibirsk, Russia}
\email{sharaf@math.nsc.ru}

\newcommand{\C}{{\mathbb C}}
\newcommand{\R}{{\mathbb R}}
\newcommand{\Z}{{\mathbb Z}}

\newcommand{\PD}{\partial}

\begin{document}

\begin{abstract}
The momentum ray transform $I_m^k$ integrates a rank $m$ symmetric tensor field $f$ on ${\R}^n$ over lines with the weight $t^k$,
$I_m^kf(x,\xi)=\int_{-\infty}^\infty t^k\langle f(x+t\xi),\xi^m\rangle\,\mathrm{d}t$.
We compute the normal operator $N_m^k=(I_m^k){}^*I_m^k$ and present an inversion formula recovering a rank $m$ symmetric tensor field $f$ from the data $(N_m^0f,\dots,N_m^mf)$.

\medskip


\end{abstract}
	\maketitle
 \section{Introduction}
Let $\langle\cdot,\cdot\rangle$ be the standard dot product on ${\R}^n$ and $|\cdot|$, the corresponding norm.

For Schwartz class functions, the ray transform (also called the X-ray transform) is defined by
\begin{equation}
If(x,\xi) = \int_{-\infty}^{\infty}f(x+t\xi)\, \mathrm{d}t
                      \label{1.1}
\end{equation}
for all $(x,\xi)\in \mathbb{R}^n \times \mathbb{R}^n$ satisfying $|\xi|=1$ and $\langle x, \xi\rangle =0.$
For a Schwartz class symmetric $m$-tensor field
$f = (f_{i_1\dots i_m})$, the ray transform is defined by
\begin{equation}
\begin{split}
I_mf(x,\xi)
& = \int_{-\infty}^{\infty} f_{i_1\dots i_m}(x+t\xi)\xi^{i_1}\dots \xi^{i_m}\, \mathrm{d}t \\
& = \int_{-\infty}^{\infty} \langle f(x+t\xi), \xi^{m}\rangle\, \mathrm{d}t.
\end{split}
                    \label{1.2}
\end{equation}
We use the Einstein summation rule to sum from $1$ to $n$ over every repeated index in lower and upper positions in a monomial.
In particular, when $m=0$, the definition \eqref{1.2} coincides with \eqref{1.1} and when $m=1$, \eqref{1.2}
represents the ray transform of vector fields which is also called the Doppler transform.

In the case of $m=0$, the ray transform $If$ uniquely determines a function $f$ and there is an explicit inversion formula. However, if $m\geq 1$, the ray transform $I_m$ has a nontrivial kernel. In particular, $I(\sigma \nabla h)=0$ whenever $h$ is a smooth
symmetric $(m-1)$-tensor field on ${\R}^n$ decaying at infinity, $\nabla$ is the total covariant derivative, and $\sigma$ denotes the symmetrization of a tensor.
A symmetric $m$-tensor field $f$ sufficiently fast decaying at infinity can be uniquely decomposed
$$
f = f^s + \sigma \nabla h,\quad h(x)\to0\ \mbox{as}\ |x|\to\infty
$$
to the solenoidal (= divergence-free) part $f^s$ and potential part $\sigma \nabla h$; see \cite[Theorem 2.6.2]{sharafutdinov2012integral}  and \cite[Theorem 6.4.7]{GIP2D} for the detailed explanation in the Euclidean case as well as in the case of Riemannian manifolds. The solenoidal part of a symmetric $m$-tensor field $f$ can be uniquely determined from $I_mf$ and there is an explicit inversion formula \cite[Theorem 2.12.2]{sharafutdinov2012integral}.

It is natural to ask: what additional information required along with $I_mf$ so that one could recover the entire tensor field $f$. This leads to the notion of the momentum ray transform $I_m^k$ introduced in \cite[Section 2.17]{sharafutdinov2012integral} by
$$
\begin{aligned}
I_m^kf(x,\xi)
& = \int_{-\infty}^{\infty} t^kf_{i_1\dots i_m}(x+t\xi)\xi^{i_1}\dots \xi^{i_m}\, \mathrm{d}t \\
& = \int_{-\infty}^{\infty} t^k\langle f(x+t\xi), \xi^{m}\rangle\, \mathrm{d}t\quad(k=0,1,2.\dots)
\end{aligned}
$$
for all $(x,\xi)\in \mathbb{R}^n \times \mathbb{R}^n$ satisfying $|\xi|=1$ and $\langle x, \xi\rangle =0$.
In particular $I_m^0 = I_m$.

A rank $m$ symmetric tensor field $f$ is uniquely determined by the data $(I_m^0f, \dots, I_m^mf)$.
This was proved in \cite[Theorem 2.17.2]{sharafutdinov2012integral}.
Later this result was extended to a Helgason type support theorem for tensor fields on a simple real analytic Riemannian manifold \cite{Abhishek:Mishra:2019}.
An algorithm for recovering $f$ from the data $\left(I_m^0 f,\dots, I_m^m f\right)$ is presented in \cite{Krishnan:Manna:Sahoo:Sharafutdinov:2019}
as well as a stability estimate in (generalized) Sobolev norms. A range characterization for the operator $f \mapsto (I_m^0f,\dots, I_m^mf)$ on the Schwartz space was established in \cite{Krishnan:Manna:Sahoo:Sharafutdinov:2020}.

Let us introduce the normal operator $N_m^k=\left(I_m^k\right)^* I_m^k$, where $\left(I_m^k\right)^*$ is the $L^2$-adjoint of the momentum ray transform $I_m^k$.
Since $N_m^k$ is an averaging operator, the data $N_m^kf$ could represent a better measurement model rather than $I^k_mf$.

We derive the inversion formula
$$
f=(-\Delta)^{1/2}\sum\limits_{k=0}^m D_{m,n}^k(N_m^kf)
$$
recovering a rank $m$ tensor field $f$ from the data $(N_m^0f,\dots,N_m^mf)$.
Here $D_{m,n}^k$ are linear differential operators, on the space of rank $m$ symmetric tensor fields, defined by an explicit formula.

The ray transform has several important applications that include X-ray computerized tomography (CT) in medical imaging when $m=0.$ In the case of $m=1$, the ray transform is used in Doppler tomography to analyze vector fields. In cases where $m=2$ or $m=4$, the ray transform and its variants are applied to tomography problems in anisotropic media regarding the elasticity and Maxwell systems, see \cite{sharafutdinov2012integral} and  \cite{Lionheart:sharafutdinov:2009,Sharafutdinov:Wang:2012}. Recently, the momentum ray transform has been adopted as a solution tool for the classical Calder\'on problem for the bi-Laplace model and other higher-order operators \cite{bhattacharyya2021unique,Sahoo:Salo:2023,bhattacharyya2023local}. The unique continuation principle for $I_m$ and $I^k_m$ is proved in \cite{Agrawal:Krishnan:Sahoo:2022}. See also \cite{ilmavirta2023unique} for a related work involving a fractional momentum operator.

\section{Basic definitions}\label{sec:pre}

First of all, mostly following \cite[Chapter 2]{sharafutdinov2012integral}, we introduce some notation and definitions concerning tensor algebra and analysis which will be used throughout the article.

\subsection{Tensor algebra over \texorpdfstring{${\R}^n$}{Rn}}\label{sec:pre:1}
Let $T^m\mathbb{R}^n$ be the $n^m$-dimensional complex vector space of $m$-tensors on $\mathbb{R}^n$. In particular, $T^0\mathbb{R}^n={\C}$ and
$T^1\mathbb{R}^n={\C}^n$. We need complex tensors since we are going to use the Fourier transform.
Assuming $n$ to be fixed, the notation $T^m\mathbb{R}^n$ will be mostly abbreviated to $T^m$.
For a fixed orthonormal basis $(e_1,\ldots, e_n)$ of $\mathbb{R}^n$, by $u_{i_1\dots i_m}=u^{i_1\dots i_m}=u(e_{i_1},\ldots, e_{i_m})$ we denote {\it coordinates} (= {\it components}) of a tensor $u\in T^m$ with respect to the basis. There is no difference between covariant and contravariant tensors since we use orthonormal bases only. Given $u\in T^m$ and $v\in T^k$, {\it the tensor product} $u\otimes v\in T^{m+k}$ is defined by
$(u\otimes v)_{i_1\dots i_{m+k}}=u_{i_1\dots i_m} v_{i_{m+1}\dots i_{m+k}}.$
The standard dot product on $\mathbb{R}^n$ extends to $T^m$ by
$\langle u,v\rangle=u^{i_1\dots i_m}\overline{v_{i_1\dots i_m}}.$
Throughout the article, the Einstein summation convention is used.

Let $S^m=S^m\mathbb{R}^n$ be the $\binom{n+m-1}{m}$-dimensional subspace of $T^m$ consisting of symmetric tensors.
{\it The partial symmetrization} $\sigma(i_1\dots i_m):T^{m+k}\to T^{m+k}$ in the indices $(i_1,\dots,i_m)$ is defined by
$$
\sigma(i_1\dots i_m) u_{i_1\dots i_mj_1\dots j_k}=\frac{1}{m !} \sum_{\pi \in \Pi_m} u_{i_{\pi(1)},\dots,i_{\pi(m)}j_1\dots j_k},
$$
where the summation is performed over the group $\Pi_m$ of all substitutions of the set $\{1,\dots,m\}$. In particular,
$\sigma: T^m\rightarrow S^m$ is the symmetrization in all indices.
Given $u\in S^m$ and $v\in S^k$, {\it the symmetric product} $u v\in S^{m+k}$ is defined by $uv=\sigma(u\otimes v)$. Being equipped with the symmetric product, $S^*{\R}^n=\bigoplus\limits_{m=0}^\infty S^m{\R}^n$ becomes a commutative graded algebra that is called {\it the algebra of symmetric tensors over} ${\R}^n$. The algebra $S^*{\R}^n$ is canonically isomorphic to the algebra of polynomials on ${\R}^n$. Every statement on symmetric tensors can be translated to the langauge of polynomials, and {vice} versa.

Given $u\in S^m$, let $i_u:S^k\to S^{m+k}$ be the operator of symmetric multiplication by $u$ and let {$j_u:S^{m+k}\to S^k$} be the adjoint of $i_u$. These {operators} are written in coordinates as
\begin{align*}
\left(i_u v\right)_{i_1 \ldots i_{m+k}} & =\sigma\left(i_1 \ldots i_{m+k}\right) u_{i_1 \ldots i_m} v_{i_{m+1} \ldots i_{m+k}} \\
\left(j_u v\right)_{i_1 \ldots i_k} & =v_{i_1 \ldots i_{m+k}} u^{i_{k+1} \ldots i_{m+k}}.
\end{align*}
The tensor $j_uv$ will be also denoted by $v/u$. For the Kronecker tensor $\delta$, the notations $i_\delta$ and $j_\delta$ will be abbreviated to $i$ and $j$ respectively.

\subsection{Tensor fields}
Recall that the Schwartz space $\mathcal{S}\left(\mathbb{R}^n\right)$ is the topological vector space consisting of $C^\infty$-smooth complex-valued functions on ${\R}^n$ fast decaying at infinity together with all derivatives, equipped with the standard topology.
Let $\mathcal{S}\left(\mathbb{R}^n; S^m\right)=\mathcal{S}\left(\mathbb{R}^n\right)\otimes S^m$ be the topological vector space of smooth fast decaying  symmetric $m$-tensor fields, defined on $\mathbb{R}^n$, whose components belong to the Schwartz space. In Cartesian coordinates, such a tensor field is written as $f=(f_{i_1\dots i_m})$ with coordinates (= components) $f_{i_1\dots i_m}=f^{i_1\dots i_m}\in\mathcal{S}\left(\mathbb{R}^n\right)$ symmetric in all indices. We again emphasize that there is no difference between covariant and contravariant coordinates since we use Cartesian coordinates only.

We use {\it the Fourier transform}
$\mathcal{F}: \mathcal{S}(\mathbb{R}^n) \rightarrow \mathcal{S}(\mathbb{R}^n),$ $f \mapsto \widehat{f}$
in the form (hereafter \textsl{i} is the imaginary unit)
\[
\mathcal{F}{f}(y)=\frac{1}{(2 \pi)^{n / 2}} \int_{\mathbb{R}^n} e^{-\textsl{i}\langle y, x\rangle} f(x)\, \mathrm{d} x.
\]
The Fourier transform $\mathcal{F}: \mathcal{S}\left(\mathbb{R}^n; S^m\right) \rightarrow \mathcal{S}\left(\mathbb{R}^n; S^m\right)$, $f \mapsto \widehat{f}$ of symmetric tensor fields is defined component-wise:
$\widehat{f}_{i_1 \ldots i_m}=\widehat{f_{i_1 \ldots i_m}}$.

Besides $\mathcal{S}\left({\R}^n; S^m\right)$, we use some other spaces of tensor fields. In particular,
$C^\infty\left(U; T^m\right)$ is the space of smooth $m$-tensor fields on an open set $U\subset{\R}^n$.
See details in
\cite[Section 2.1]{sharafutdinov2012integral}.

The $L^2$-product on $C_0^\infty\left({\R}^n; T^m\right)$ is defined by
\begin{equation}
(f,g)_{L^2({\R}^n; T^m)}=\int_{{\R}^n}\langle f(x),g(x)\rangle\,dx.
                               \label{2.1}
\end{equation}

\subsection{Inner derivative and divergence}
The first order differential operator
$$
d:C^\infty({\R}^n;S^m)\to C^\infty({\R}^n;S^{m+1})
$$
defined by
$$
(df)_{i_1\dots i_{m+1}}=\sigma(i_1\dots i_{m+1})\frac{\partial f_{i_1\dots i_m}}{\partial x^{i_{m+1}}}
=\frac{1}{m+1}\Big(\frac{\partial f_{i_2\dots i_{m+1}}}{\partial x^{i_1}}+\dots
+\frac{\partial f_{i_1\dots i_m}}{\partial x^{i_{m+1}}}\Big)
$$
is called {\it the inner derivative}.

{\it The divergence}
$$
\mbox{div}:C^\infty({\R}^n;S^{m+1})\to C^\infty({\R}^n;S^m)
$$
is defined by
$$
(\mbox{div}\,f)_{i_1\dots i_m}=\delta^{jk}\frac{\partial f_{i_1\dots i_mj}}{\partial x^k}.
$$
The operators $d$ and $-\mbox{div}$ are formally adjoint to each other with respect to the $L^2$-product \eqref{2.1}. The divergence is denoted by $\delta$ in \cite{sharafutdinov2012integral}. But we will always use the notation $\mbox{div}$ since some our formulas involve the divergence and Kronecker tensor simultaneously.

\subsection{\texorpdfstring{The space ${\mathcal S}(T{\mathbb S}^{n-1})$}{The space S(TS^{n-1})}}

The Schwartz space ${\mathcal S}(E)$ is well defined for a smooth vector bundle $E\to M$ over a compact manifold using a finite atlas and a partition of unity subordinate to the atlas.
In particular, the Schwartz space ${\mathcal S}(T{\mathbb S}^{n-1})$ is well defined for the tangent bundle
$$
T\mathbb S^{n-1}= \{(x, \xi) \in \mathbb{R}^n \times \mathbb{S}^{n-1}:\langle x, \xi\rangle=0\}\to\mathbb{S}^{n-1},\quad(x,\xi)\mapsto\xi
$$
of the unit sphere $\mathbb S^{n-1}=\{x\in{\R}^n:|x|=1\}$.

The Fourier transform $\mathcal{F}: \mathcal{S}\left(T \mathbb{S}^{n-1}\right) \rightarrow \mathcal{S}\left(T \mathbb{S}^{n-1}\right), \varphi \mapsto \widehat{\varphi}$ is defined by
$$
\mathcal{F}{\varphi}(y, \xi)=\frac{1}{(2 \pi)^{(n-1)/2}} \int_{\xi^{\perp}} e^{-\textsl{i}\langle y, x\rangle} \varphi(x, \xi)\,\mathrm{d} x,
$$
where $\mathrm{d} x$ is the $(n-1)$-dimensional Lebesgue measure on the hyperplane $\xi^{\perp}=\left\{x \in \mathbb{R}^n ; \right.$ $\langle\xi, x\rangle=0\}$. Notice that it is the standard Fourier transform in the $(n-1)$-dimensional variable $x$ while $\xi \in \mathbb{S}^{n-1}$ is considered as a parameter.

The $L^2$-product on ${\mathcal S}(T{\mathbb S}^{n-1})$ is defined by
\begin{equation}
(\varphi,\psi)_{L^2(T\mathbb{S}^{n-1})}=\int\limits_{\mathbb{S}^{n-1}}\int\limits_{\xi^\bot}
\varphi(x,\xi)\overline{\psi(x,\xi)}\,\mathrm{d} x\,\mathrm{d}\xi,
                              \label{2.1a}
\end{equation}
where $\mathrm{d}\xi$ is the $(n-1)$-dimensional Euclidean volume form on the unit sphere $\mathbb{S}^{n-1}$.

\subsection{Ray transforms}
It is convenient to parameterize the family of oriented lines in ${\R}^n$ by points of the manifold $T \mathbb{S}^{n-1}$. Namely, a point
$(x,\xi)\in T \mathbb{S}^{n-1}$ determines the line $\{x+t\xi:t\in{\R}\}$ through $x$ in the direction $\xi$.

{\it The ray transform}
$$
I_m:\mathcal{S}({\R}^n;S^m)\to\mathcal{S}\left(T \mathbb{S}^{n-1}\right)
$$
is the linear continuous operator defined by
$$
I_m f(x, \xi)=\int_{\mathbb{R}} f_{i_1 \dots i_m}(x+t \xi) \xi^{i_1} \dots \xi^{i_m} \mathrm{~d} t=\int_{\mathbb{R}}\left\langle f(x+t \xi), \xi^m\right\rangle \mathrm{d} t.
$$
The ray transform is related to the Fourier transform by the important formula \cite[formula 2.1.15]{sharafutdinov2012integral}
\begin{equation}
\widehat{I_mf}(y,\xi)=(2\pi)^{1/2}\langle \widehat f(y),\xi^m\rangle\quad\big((y,\xi)\in T \mathbb{S}^{n-1}\big),
                              \label{2.2}
\end{equation}
where $\widehat{I_mf}$ is the Fourier transform of the function $I_mf\in\mathcal{S}\left(T \mathbb{S}^{n-1}\right)$.

For $0\le k\le m$, {\it the momentum ray transform}
$$
I_m^k:\mathcal{S}({\R}^n;S^m)\to\mathcal{S}\left(T \mathbb{S}^{n-1}\right)
$$
is the linear continuous operator defined by
\begin{equation}
I_m^k f(x, \xi)=\int_{\mathbb{R}}t^k f_{i_1 \dots i_m}(x+t \xi) \xi^{i_1} \dots \xi^{i_m} \mathrm{~d} t
=\int_{\mathbb{R}}t^k\left\langle f(x+t \xi), \xi^m\right\rangle \mathrm{d} t.
                                  \label{2.3}
\end{equation}
The formula \eqref{2.2} is generalized as follows \cite[formula (2.9)]{Krishnan:Manna:Sahoo:Sharafutdinov:2019}:
$$
\widehat{I_m^kf}(y,\xi)=(2\pi)^{1/2}{\textsl i}^k\left\langle d^k\widehat f(y),\xi^{m+k}\right\rangle
\quad\big((y,\xi)\in T \mathbb{S}^{n-1}\big).
$$
As we will see later, $I_m^k$ should be considered together with lower degree operators $I_m^0,\dots, I_m^{k-1}$, i.e., the collection
$(I_m^0f,\dots, I_m^kf)$ represents more convenient information about $f$ than $I_m^kf$ alone.

\subsection{Normal operators}
The formal adjoint of the ray transform $I_m$ with respect to $L^2$-products  \eqref{2.1} and  \eqref{2.1a}
\[
I_m^*: \mathcal{S}\left(T \mathbb{S}^{n-1}\right) \rightarrow C^\infty\left(\mathbb{R}^n ; S^m\right)
\]
is expressed by
\[
\left(I_m^*\varphi\right)_{i_1 \dots i_m} (x)=\int_{\mathbb{S}^{n-1}} \xi_{i_1} \dots \xi_{i_m}
\varphi\big(x-\langle x, \xi\rangle \xi, \xi\big)\, \mathrm{d}\xi.
\]
We emphasize that, for $\varphi \in \mathcal{S}(T \mathbb{S}^{n-1})$, the tensor field $I_m^*\varphi$ does not need to decay fast at infinity. 

Similarly, the formal $L^2$-adjoint of the momentum ray transform
$I_m^k$
\[
\left(I_m^k\right)^*: \mathcal{S}\left(T \mathbb{S}^{n-1}\right) \rightarrow C^\infty\left(\mathbb{R}^n ; S^m\right)
\]
is expressed by
\begin{equation}
\big((I_m^k)^*\varphi\big)_{i_1 \ldots i_m}(x)
=\int_{\mathbb{S}^{n-1}}\langle x, \xi\rangle^k \xi_{i_1} \dots \xi_{i_m} \varphi\big(x-\langle x, \xi\rangle \xi, \xi\big)\, \mathrm{d}\xi.
                                  \label{2.4}
\end{equation}
Let
$$
N_m=I_m^*I_m: \mathcal S\left(\mathbb{R}^n ; S^m\right) \rightarrow C^\infty\left(\mathbb{R}^n ; S^m\right)
$$
be the normal operator for the ray transform $I_m$.
Similarly, let
$$
N_m^k=(I_m^k)^*I_m^k: \mathcal S\left(\mathbb{R}^n ; S^m\right) \rightarrow C^\infty\left(\mathbb{R}^n ; S^m\right)
$$
be the normal operator for the momentum ray transform $I_m^k$.

Given $f\in\mathcal S\left(\mathbb{R}^n ; S^m\right)$, the tensor field $N_m^kf$ does not grow too fast at infinity, i.e., the estimate
$$
|N_m^kf(x)|\le C(1+|x|)^N
$$
holds with some constants $C$ and $N$. In particular, $N_m^kf$ can be considered as a tempered tensor field-distribution, i.e.,
$N_m^kf\in{\mathcal S}'\left(\mathbb{R}^n ; S^m\right)$. Hence the Fourier transform
$\widehat{N_m^kf}\in{\mathcal S}'\left(\mathbb{R}^n ; S^m\right)$ is well defined at least in the distribution sense. We will show that, for $f\in\mathcal S\left(\mathbb{R}^n ; S^m\right)$, the restriction of $\widehat{N_m^kf}$ to ${\R}^n\setminus\{0\}$ belongs to
$C^\infty\left(\mathbb{R}^n\setminus\{0\}; S^m\right)$.

The operator $N_m$ was computed in \cite[formula 2.11.3]{sharafutdinov2012integral} where the notation $\mu^m$ was used instead of $I_m^*$.
In Section 4, we will derive a similar formula for $N_m^k$. In \cite{Agrawal:Krishnan:Sahoo:2022}, a similar expression for the normal operator is considered to study the unique continuation principle for momentum ray transforms.

\section{The inversion formula}

Let $\Gamma$ be Euler's Gamma function and let the operator $(-\Delta)^{1/2}$ be defined defined via Fourier transform by $|y|{\mathcal F}={\mathcal F}(-\Delta)^{1/2}$.
We use the definition
$$
(2l+1)!!=1\cdot3\cdots(2l+1),\quad (-1)!!=1.
$$

\begin{theorem} \label{Th3.1}
Given integers $m\ge0$ and $n\ge2$,
 a tensor field $f\in \mathcal{S}\left(\mathbb{R}^n; S^m\right)$ is recovered from the data $(N_m^0f,\dots,N_m^mf)$
by the inversion formula
\begin{equation}
f(x)=(-\Delta)^{1/2}\sum\limits_{k=0}^m D_{m,n}^k(N_m^kf)(x),
                           \label{3.1}
\end{equation}
where the linear differential operator of order $m+k$
$$
D_{m,n}^k:C^\infty({\R}^n;S^m)\to C^\infty({\R}^n;S^m)
$$
is defined by
\begin{align}\label{3.2}
    \begin{split}
        D_{m,n}^k&=c_{m,n}^k\sum\limits_{p=k}^m(n\!+\!2m\!-\!2p\!-\!3)!!
\!\!\!\\&\qquad\qquad\times\sum\limits_{q=0}^{\min(p,m-p,p-k)}\!\!\!\!
\frac{(-1)^q}{2^q q! (m\!-\!p\!-\!q)! (p\!-\!k\!-\!q)!}
\, d^{p-q}\, i^q\, j^q\, j_x^{p-k-q}\, \mbox{\rm div}^k
    \end{split}
\end{align}
with the coefficient
\begin{equation}
c_{m,n}^k=\frac{(-1)^k}{(k!)^2}\,\frac{2^{m-2}\, \Gamma\big(\frac{2m+n-1}{2}\big)}{\pi^{(n+1)/2}\, (n+2m-3)!!}
                           \label{3.3}
\end{equation}
and the operators \(i\), \(j\), and \(j_x\) are defined in Section \ref{sec:pre:1}.
\end{theorem}

The appearance of $(-\Delta)^{1/2}$ in the inversion formula  \eqref{3.1} needs some justification. Indeed, according to the definition
$(-\Delta)^{1/2}u={\mathcal F}^{-1}(|y|\widehat u)$, the tensor field $(-\Delta)^{1/2}u$ is well defined if $|y|\widehat u(y)$ belongs to the domain of the Fourier transform. We will show in the next section that the tensor field $u=\sum_kD_{m,n}^k(N_m^kf)$ is such that $|y|\widehat u$ can be interpreted as a tempered distribution. Therefore the right-hand side of \eqref{3.1} is well defined.

Let us write down the inversion formula \eqref{3.1} in the cases of $m=1$ and $m=2$.

For a vector field $f\in \mathcal{S}\left(\mathbb{R}^n; {\C}^n\right)\ (n\ge2)$,
$$
f(x)=\frac{\Gamma\big(\frac{n+1}{2}\big)}{2\pi^{(n+1)/2}}
(-\Delta)^{1/2}\Big(N_1^0f+\frac{1}{n\!-\!1}\,d\,j_x\,(N_1^0f)
-\frac{1}{n\!-\!1}\,\,d\,\mbox{\rm div}\,(N_1^1f)\Big).
$$

For a second rank tensor field $f\in \mathcal{S}(\mathbb{R}^n;S^2)$,
$$
\begin{aligned}
                f(x)=\frac{\Gamma\big(\frac{n+3}{2}\big)}{2 \pi^{(n+1)/2}}\,(-\Delta)^{1/2}\bigg[&
                N_2^0f-\frac{1}{n+1}\,ij\,N_2^0f\\
                &+\frac{2}{n+1}\,d\Big(j_x\,N_2^0f-\mbox{div}\,N_2^1f\Big)\\
                &+\frac{1}{(n\!-\!1)(n\!+\!1)}\,d^2\Big(
                j_x^2N_2^0f
                -2\,j_x\,\mbox{div}\,N_2^1f
                +\frac{1}{2}\,\mbox{div}^2\,N_2^2f
                \Big)\bigg].
\end{aligned}
$$

The rest of the article is mostly devoted to the proof of Theorem \ref{Th3.1}. Now, we discuss the scheme of the proof.

We introduce the tensor fields
\begin{equation}
A^{(m,k)}\in C^\infty({\R}^n\setminus\{0\};S^{2m-k})\quad(0\le k\le m)
                           \label{3.4}
\end{equation}
by
\begin{equation}
A^{(m,k)}(y)=d^{2m-k}|y|^{2m-2k-1}.
                           \label{3.5}
\end{equation}
These tensor fields play an important role in all our arguments.

In the next section, we compute the normal operators $N_m^k\ (0\le k\le m)$ and show that a tensor field $f\in{\mathcal S}({\R}^n;S^m)$ satisfies $A^{(m,0)}/(d^k\widehat f)=F^{(m,k)}$, where the tensor field
$F^{(m,k)}\in C^\infty({\R}^n\setminus\{0\};S^{m-k})$ is defined by
\begin{equation}
F^{(m,k)}(y)=\frac{2^{m-2}(2m-1)!!\,\Gamma\big(\frac{2m\!+\!n\!-\!1}{2}\big)}{\pi^{(n+1)/2}\,k!}\,j_y^k\,\widehat{N_m^kf}(y).
                           \label{3.6}
\end{equation}
To avoid proliferation of the $\widehat{\cdot}$ symbol, we denote $g(y)=\widehat f(y)$
and write the equation $A^{(m,0)}/(d^k\widehat f)=F^{(m,k)}$ in the form
\begin{equation}
A^{(m,0)}/(d^kg)=F^{(m,k)}\quad(0\le k\le m).
                           \label{3.7}
\end{equation}
Given the data $(F^{(m,0)},\dots,F^{(m,m)})$, we consider \eqref{3.7} as a system of linear equations for the unknown tensor field $g$.

The first equation of the system  \eqref{3.7}
$$
A^{(m,0)}(y)/g(y)=F^{(m,0)}(y)\quad(y\in{\R}^n\setminus\{0\})
$$
is a pure algebraic equation. More precisely, being written in coordinates, it constitutes a system of linear algebraic equations in the components of the tensor $g(y)$ with coefficients depending on $y$. The system was considered in \cite[Theorem 2.12.1]{sharafutdinov2012integral}
where the tensor field
$\varepsilon^m(y)=\frac{|y|}{((2m-1)!!)^2}\,d^{2m}|y|^{2m-1}$ was used instead of $A^{(m,0)}$. It allows to determine the {\it tangential part} of the tensor field $g$ which corresponds to the solenoidal part of $f={\mathcal F}^{-1}g$ (see  \cite[Section 2.6]{sharafutdinov2012integral} for the definition of the tangential part).

The second equation of the system \eqref{3.7}, $A^{(m,1)}/(dg)=F^{(m,1)}$, constitutes a system of linear first order PDEs in components of the tensor field $g$, the third equation constitutes a system of linear second order PDEs, and so on.

The system \eqref{3.7} can be reduced to the purely algebraic system
\begin{equation}
A^{(m,k)}/g=H^{(m,k)}\quad(0\le k\le m)
                           \label{3.8}
\end{equation}
with the right-hand side defined by
\begin{equation}
H^{(m,k)}=\frac{(2m-2k-1)!!(n+2m-2k-3)!!}{(2m-1)!!(n+2m-3)!!}
\sum\limits_{p=0}^k(-1)^p\binom{k}{p} \,\text{div}^{k-p}\,F^{(m,p)}.
                           \label{3.9}
\end{equation}
The reduction is presented in Section 5. The precise sense of the reduction is expressed by Proposition \ref{P5.2} below, see also the paragraph after Proposition \ref{P5.2}.

Some consistency conditions should be imposed on right-hand side $H^{(m,k)}$ for solvability  of the system \eqref{3.9}. In the case of a general $m$, it is not easy to write down the consistency conditions explicitly. Fortunately, we do not need to know the consistency conditions; in our setting, the system \eqref{3.9} has a solution by Propositions \ref{P4.4} and \ref{P5.2} presented below. If the system \eqref{3.9} has a solution, then the solution is unique and is expressed by \eqref{3.5} with $\widehat f=g$. This fact is proved in Section~6.

\section{Normal operator}\label{sec:norm}

We start with the proof of \eqref{2.4}. For $f\in{\mathcal S}({\R}^n;S^m)$ and $\varphi\in{\mathcal S}(T{\mathbb S}^{n-1})$,
\begin{equation}
\begin{aligned}
(I_m^kf,\varphi)_{L^2(T{\mathbb S}^{n-1})}
&=\int\limits_{{\mathbb S}^{n-1}}\int\limits_{\xi^\bot}(I_m^kf)(x,\xi)\,\overline{\varphi(x,\xi)}\,{\mathrm d}x\,{\mathrm d}\xi\\
&=\int\limits_{{\mathbb S}^{n-1}}\int\limits_{\xi^\bot}\int\limits_{-\infty}^\infty
t^k\langle f(x'+t\xi),\xi^m\rangle\,\overline{\varphi(x',\xi)}\,{\mathrm d}t\,{\mathrm d}x'\,{\mathrm d}\xi.
\end{aligned}
                           \label{4.1}
\end{equation}
We transform the inner integral by the change $x=x'+t\xi$ of integration variables
$$
\begin{aligned}
\int\limits_{\xi^\bot}\int\limits_{-\infty}^\infty
t^k\langle f(x'+t\xi),\xi^m\rangle\,\overline{\varphi(x',\xi)}\,&{\mathrm d}t\,{\mathrm d}x'
=\int\limits_{{\R}^n}\langle x,\xi\rangle^k\langle f(x),\xi^m\rangle\,
\overline{\varphi(x-\langle x,\xi\rangle\xi,\xi)}\,{\mathrm d}x\\
&=\int\limits_{{\R}^n}\langle x,\xi\rangle^k f^{i_1\dots i_m}(x)\,\xi_{i_1}\dots\xi_{i_m}\,
\overline{\varphi(x-\langle x,\xi\rangle\xi,\xi)}\,{\mathrm d}x.
\end{aligned}
$$
Substituting this expression into \eqref{4.1}, we obtain
$$
\begin{aligned}
(I_m^kf,\varphi)_{L^2(T{\mathbb S}^{n-1})}
&=\int\limits_{{\R}^n} f^{i_1\dots i_m}(x)\,
\overline{\int\limits_{{\mathbb S}^{n-1}}\langle x,\xi\rangle^k\xi_{i_1}\dots\xi_{i_m}\,
\varphi(x-\langle x,\xi\rangle\xi,\xi)\,{\mathrm d}\xi}\,{\mathrm d}x\\
&=(f,(I_m^k)^*\varphi)_{L^2({\R}^n;S^m)}.
\end{aligned}
$$
This proves \eqref{2.4}.

Recall that $N_m^k=(I_m^k)^*I_m^k$ is the normal operator for the momentum ray transform.

\begin{proposition} \label{P4.1}
Let $0\leq k\leq m$ and $n\ge2$.
For a tensor field $f\in \mathcal{S}\left(\mathbb{R}^n; S^m\right)$,
\begin{equation}
(N_m^kf)_{i_1\dots i_m}(x) =2\sum_{l=0}^k\binom{k}{l}
(x^{k+l}f)^{j_1\dots j_mp_1\dots p_{k+l}} \ast
\frac{(x^{2m+k+l})_{i_1 \dots i_mj_1 \dots j_mp_1\dots p_{k+l}}}{|x|^{2m+2l+n-1}},
                           \label{4.2}
\end{equation}
where $\ast$ denotes the convolution.
\end{proposition}

The right-hand side of \eqref{4.2} needs the following comment. For $x\in{\R}^n$, according to our definition of the symmetric product,
$x^{k+l}\in S^{k+l}$ with coordinates $(x^{k+l})^{p_1\dots p_{k+l}}=x^{p_1}\dots x^{p_{k+l}}$. Therefore, for $f\in S^m$,
$$
(x^{k+l}f)^{j_1\dots j_mp_1\dots p_{k+l}}=
\sigma(j_1\dots j_mp_1\dots p_{k+l})(x^{p_1}\dots x^{p_{k+l}}f^{j_1\dots j_m}).
$$

Before proving Proposition \ref{P4.1}, we observe that it implies some regularity of the tensor field $N_m^kf$. Indeed, the first factor
$(x^{k+l}f)^{j_1\dots j_mp_1\dots p_{k+l}}$ on the right-hand side of \eqref{4.2} belongs to ${\mathcal S}({\R}^n)$. The second factor is a function locally summable over ${\R}^n$ and bounded for $|x|\ge1$. Hence the second factor can be considered as an element of the space ${\mathcal S}'({\R}^n)$ of tempered distributions. As well known \cite{Vlad79}, for $u\in{\mathcal S}({\R}^n)$ and
$v\in{\mathcal S}'({\R}^n)$, the convolution $u\ast v$ is defined and belongs to the space of smooth functions whose every derivative increases at most as a polynomial at infinity. In this case the standard formula is valid: $\widehat{u\ast v}=(2\pi)^{n/2}\,\widehat u\widehat v$. Thus, we can state that
$$
N_m^k:\mathcal{S}\left(\mathbb{R}^n; S^m\right)\to C^\infty\left(\mathbb{R}^n; S^m\right)
$$
is a continuous operator.

To prove Proposition \ref{P4.1} we need the following

\begin{lemma}   \label{L4.1}
Let $k\ge0$ be an integer, $0\neq a\in\R$, and $b\in\R$. Then
\begin{align*}
    \sum_{l=0}^k(-1)^l\binom{k}{l} \frac{(a^2+b)^{2k-l}}{a^{2k-2l}}=\sum_{l=0}^k\binom{k}{l}\frac{b^{k+l}}{a^{2l}}.
\end{align*}
\end{lemma}

\begin{proof}
By the binomial formula,
\begin{align*}
\sum_{l=0}^k (-1)^l\binom{k}{l} \frac{(a^2+b)^{2k-l}}{a^{2k-2l}}
&=\frac{1}{a^{2k}}\sum_{l=0}^k \binom{k}{l} (-a^2)^l(a^2+b)^{2k-l}\\
&=\frac{1}{a^{2k}}\sum_{l=0}^k \binom{k}{l} (-a^2+b-b)^l(a^2+b)^{2k-l}\\
&=\frac{(a^2+b)^k}{a^{2k}}\sum_{l=0}^k \binom{k}{l} (-a^2+b-b)^l(a^2+b)^{k-l}\\
&=\frac{b^k\left(a^2+b\right)^k}{a^{2k}}
= \sum_{l=0}^k\binom{k}{l}\frac{b^{k+l}}{a^{2l}}.\qedhere
\end{align*}
\end{proof}

\begin{proof}[Proof of Proposition \ref{P4.1}]
Using \eqref{2.3} and \eqref{2.4}, we first compute
\begin{align*}
\left(N_m^k f\right)_{i_1 \ldots i_m}(x) &=\left(I_m^k\right)_{i_1 \dots i_m}^* I_m^k f(x) \\
& =\int_{\mathbb{S}^{n-1}}\langle x, \xi\rangle^k \xi_{i_1} \ldots \xi_{i_m} (I_m^k f)(x-\langle x, \xi\rangle \xi, \xi)\, \mathrm{d}\xi\\
&=\int_{\mathbb{S}^{n-1}}\int_{\R}t^k\langle x, \xi\rangle^k f^{j_1 \ldots j_m}(x-\langle x, \xi\rangle \xi+t\xi)\xi_{j_1}\dots \xi_{j_m}  \xi_{i_1} \ldots \xi_{i_m}\, \mathrm{d}t\mathrm{d}\xi\\
&=2\int_{\mathbb{S}^{n-1}}\int_{0}^{\infty}t^k\langle x, \xi\rangle^k f^{j_1 \ldots j_m}(x-\langle x, \xi\rangle \xi+t\xi)
(\xi^{2m})_{i_1\dots i_mj_1\dots j_m}\, \mathrm{d} t\mathrm{d}\xi.
\end{align*}
Replacing $t-\langle x, \xi\rangle$ by $t$ in the last integral, we have
\begin{align*}
\left(N_m^k f\right)_{i_1 \ldots i_m}(x)
&=2\int_{\mathbb{S}^{n-1}}\int_0^\infty\big(t+\langle x, \xi\rangle\big)^k\langle x, \xi\rangle^k
f^{j_1 \ldots j_m}(x+t\xi)(\xi^{2m})_{i_1\dots i_mj_1\dots j_m}\, \mathrm{d}t\mathrm{d}{\xi}\\
&=2\sum_{l=0}^k\binom{k}{l} \int_{\mathbb{S}^{n-1}} \int_0^\infty t^l\langle x, \xi\rangle^{2 k-l}
f^{j_1 \ldots j_m}(x+t \xi) (\xi^{2m})_{i_1\dots i_mj_1\dots j_m}\, \mathrm{d} t \mathrm{d}{\xi}.
\end{align*}
Changing integration variables by
$$
x+t\xi=z,\quad t=|z-x|,\quad\xi = \frac{z-x}{|z-x|},\quad \mathrm{d}t\mathrm{d}\xi=|z-x|^{1-n}\mathrm{d}z,
$$
we obtain
$$
\left(N_m^k f\right)_{i_1 \ldots i_m}(x)
=2 \sum_{l=0}^k\binom{k}{l} \int_{\mathbb{R}^n}\langle x, z\!-\!x\rangle^{2k-l}
\big((z\!-\!x)^{2m}\big)_{i_1 \ldots i_m j_1 \ldots j_m}
\frac{f^{j_1 \ldots j_m}(z)}{|z-x|^{2m+2k-2l+n-1}}\,\mathrm{d}z.
$$
Let us write this in the form
$$
\begin{aligned}
&\left(N_m^k f\right)_{i_1 \ldots i_m}(x)\\
&=2\int_{\mathbb{R}^n}\bigg[ \sum_{l=0}^k(-1)^l\binom{k}{l} \frac{\langle x, x-z\rangle^{2k-l}}{|x-z|^{2k-2l}}\bigg]
\big((z\!-\!x)^{2m}\big)_{i_1 \ldots i_m j_1 \ldots j_m}
\frac{f^{j_1 \ldots j_m}(z)}{|z-x|^{2m+n-1}}\,\mathrm{d}z.
\end{aligned}
$$
By Lemma \ref{L4.1} with $a=|x-z|$ and $b=\langle z,x-z\rangle$,
$$
\begin{aligned}
\sum_{l=0}^k(-1)^l\binom{k}{l} \frac{\langle x, x-z\rangle^{2k-l}}{|x-z|^{2k-2l}}
&=\sum_{l=0}^k(-1)^l\binom{k}{l} \frac{\big(|x-z|^2+\langle z, x-z\rangle\big)^{2k-l}}{|x-z|^{2k-2l}}\\
&=\sum_{l=0}^k\binom{k}{l} \frac{\langle z, x-z\rangle^{k+l}}{|x-z|^{2l}}
=\sum_{l=0}^k\binom{k}{l} \frac{\langle z, x-z\rangle^{k+l}}{|z-x|^{2l}}.
\end{aligned}
$$
Substituting this expression into the previous formula, we obtain
$$
\left(N_m^k f\right)_{i_1 \ldots i_m}(x)
=2\sum_{l=0}^k\binom{k}{l} \int_{\mathbb{R}^n}\langle z, x-z\rangle^{k+l}
\big((x\!-\!z)^{2m}\big)_{i_1 \ldots i_m j_1 \ldots j_m}
\frac{f^{j_1 \ldots j_m}(z)}{|x-z|^{2m+2l+n-1}}\,\mathrm{d}z.
$$
Then we represent the first factor of the integrand as follows
$$
\langle z, x-z\rangle^{k+l}=(z^{k+l})^{p_1\dots p_{k+l}}\big((x-z)^{k+l}\big)_{p_1\dots p_{k+l}}.
$$
Substituting this expression into the previous formula, we write the result in the form
$$
\begin{aligned}
&\left(N_m^k f\right)_{i_1 \ldots i_m}(x)\\
&=2\sum_{l=0}^k\binom{k}{l} \int_{\mathbb{R}^n}
\big(z^{k+l}\otimes f(z)\big)^{p_1\dots p_{k+l}j_1 \dots j_m}
\frac{\big((x\!-\!z)^{2m+k+l}\big)_{i_1 \ldots i_m j_1 \ldots j_mp_1\dots p_{k+l}}}{|x-z|^{2m+2l+n-1}}\,\mathrm{d}z.
\end{aligned}
$$
We can replace $\big(z^{k+l}\otimes f(z)\big)^{p_1\dots p_{k+l}j_1 \dots j_m}$ with
$\big(z^{k+l} f(z)\big)^{j_1 \dots j_mp_1\dots p_{k+l}}$ since the second factor in the integrand is symmetric in all indices. Hence
$$
\begin{aligned}
&\left(N_m^k f\right)_{i_1 \ldots i_m}(x)\\
&=2\sum_{l=0}^k\binom{k}{l} \int_{\mathbb{R}^n}
\big(z^{k+l}f(z)\big)^{j_1 \dots j_mp_1\dots p_{k+l}}
\frac{\big((x\!-\!z)^{2m+k+l}\big)_{i_1 \ldots i_m j_1 \ldots j_mp_1\dots p_{k+l}}}{|x-z|^{2m+2l+n-1}}\,\mathrm{d}z.
\end{aligned}
$$
Every integral on the right-hand side is the convolution of $\big(x^{k+l}f\big)^{j_1 \dots j_mp_1\dots p_{k+l}}$ with\hfill\\
$\frac{(x^{2m+k+l})_{i_1 \ldots i_m j_1 \ldots j_mp_1\dots p_{k+l}}}{|x|^{2m+2l+n-1}}$. We thus arrive at \eqref{4.2}.
\end{proof}

We use the abbreviated notation $\partial_{i_1\dots i_k}=\frac{\partial^k}{\partial y^{i_1}\dots\partial y^{i_k}}$ for partial derivatives.
Recall that indices can be written either in lower position or in upper position. In particular, $\PD^{i_1\dots i_k}=\PD_{i_1\dots i_k}$.
Recall that $j_y:S^m\to S^{m-1}$ is the operator of contraction with $y$, see Subsection 2.1 where the operator $j_u$ is defined. For
$0\le k\le m$, tensor fields $A^{(m,k)}\in C^\infty({\R}^n\setminus\{0\};S^{2m-k})$ are defined by \eqref{3.5} or in coordinates
$$
A^{(m,k)}_{i_1\dots i_{2m-k}}=\partial_{i_1\dots i_{2m-k}}|y|^{2m-2k-1}.
$$

Proposition \ref{P4.1} can be equivalently written in terms of Fourier transforms $\widehat f$ and $\widehat{N_m^kf}$.

\begin{proposition} \label{P4.4}
Let $0\leq k\leq m$ and $n\ge2$. For $f\in \mathcal{S}\left(\mathbb{R}^n; S^m\right)$, the equation \eqref{3.7} holds with $g=\hat f$ and $F^{(m,k)}$ defined by \eqref{3.6}.
\end{proposition}

\begin{proof}
Applying the Fourier transform to the equality \eqref{4.2}, we obtain
$$
\widehat{N_m^kf}_{i_1\dots i_m} =2(2\pi)^{n/2}\sum_{l=0}^k\binom{k}{l}
\big(\widehat{x^{k+l}f}\big)^{j_1\dots j_mp_1\dots p_{k+l}}
{\mathcal F}\Big[\frac{(x^{2m+k+l})_{i_1 \dots i_mj_1 \dots j_mp_1\dots p_{k+l}}}{|x|^{2m+2l+n-1}}\Big].
$$
Using the standard properties of the Fourier transform \cite[Lemma 7.1.2]{Hoermander:1983}
$$
\widehat{x_jf}=\textsl{i}\,\partial_j\widehat f,\quad
\widehat{\partial_jf}=\textsl{i}\,y_j\widehat f,
$$
we transform our formula to the form
$$
\begin{aligned}
\widehat{N_m^kf}_{i_1\dots i_m} =2(2\pi)^{n/2}(-1)^{m\!+\!k}\sum_{l=0}^k(-1)^l\binom{k}{l}&
\sigma(j_1\dots j_mp_1\dots p_{k+l})\big(\partial^{p_1\dots p_{k+l}}{\widehat f}^{j_1\dots j_m}\big)\\
&\times\partial_{i_1 \dots i_mj_1 \dots j_mp_1\dots p_{k+l}}{\mathcal F}\Big[|x|^{-2m-2l-n+1}\Big].
\end{aligned}
$$
Here we can omit the symmetrization $\sigma(j_1\dots j_mp_1\dots p_{k+l})$ since the second factor \hfill\\
$\partial_{i_1 \dots i_mj_1 \dots j_mp_1\dots p_{k+l}}{\mathcal F}\Big[|x|^{-2m-2l-n+1}\Big]$ is symmetric in all indices. Thus,
\begin{equation}
\begin{aligned}
\widehat{N_m^kf}_{i_1\dots i_m}
&=2(2\pi)^{n/2}(-1)^{m\!+\!k}\sum_{l=0}^k(-1)^l\binom{k}{l}\\
&\times\big(\partial^{p_1\dots p_{k+l}}{\widehat f}^{j_1\dots j_m}\big)
\partial_{i_1 \dots i_mj_1 \dots j_mp_1\dots p_{k+l}}{\mathcal F}\Big[|x|^{-2m-2l-n+1}\Big].
\end{aligned}
                           \label{4.4}
\end{equation}

Let ${\mathcal S}'({\R}^n)$ be the space of tempered distributions.
Recall that $\lambda\mapsto|x|^\lambda$ is the meromorphic
${\mathcal S}'({\R}^n)$-valued function of $\lambda\in\C$ with simple poles at points $-n,-n-2,-n-4,\dots$. For a detailed discussion, \cite[Chapter VII, Lemmas 6.1 and 6.2]{Helgason:2011:book}.
The Fourier transform of $|x|^\lambda$
is expressed by
$$
\begin{aligned}
{\mathcal F}[|x|^\lambda]
&=\frac{2^{\lambda+n/2}\Gamma\Big(\frac{\lambda+n}{2}\Big)}{\Gamma(-\lambda/2)}|y|^{-\lambda-n}\quad(\lambda,-\lambda-n\notin2{\Z}^+),\\
{\mathcal F}[|x|^{2k}]&=(2\pi)^{n/2}(-\Delta)^k\delta\quad(k\in{\Z}^+),
\end{aligned}
$$
where $\delta$ is the Dirac function. As compared with formulas (42) and (43) on p. 238 of \cite{Helgason:2011:book}, there is the additional factor $(2\pi)^{-n/2}$ on right-hand side of these formulas. The factor appears since our definition of the Fourier transform is slightly different from the one in Helgason's book \cite{Helgason:2011:book}.
 
In particular,
$$
{\mathcal F}\Big[|x|^{-2m-2l-n+1}\Big]=\frac{\Gamma\Big(\frac{1-2m-2l}{2}\Big)}
{2^{2m+2l+n/2-1}\Gamma\Big(\frac{2m+2l+n-1}{2}\Big)}\,|y|^{2m+2l-1}.
$$
Substituting this value into \eqref{4.4}, we obtain
\begin{align}
 \begin{split}
     \widehat{N_m^kf}_{i_1\dots i_m}(y) =&\frac{(-1)^{m\!+\!k}(2\pi)^{n/2}}{2^{2m\!+\!n/2\!-\!2}}\sum_{l=0}^k\binom{k}{l}
\frac{(-1)^l\Gamma\Big(\frac{1\!-\!2m\!-\!2l}{2}\Big)}{2^{2l}\Gamma\Big(\frac{2m\!+\!2l\!+\!n\!-\!1}{2}\Big)}
\\
&\qquad\times\Big(\partial_{i_1 \dots i_mj_1 \dots j_mp_1\dots p_{k+l}}|y|^{2m\!+\!2l\!-\!1}\Big)\partial^{p_1\dots p_{k+l}}{\widehat f}^{j_1\dots j_m}(y).\label{4.5}
 \end{split}
\end{align}

Let us contract the equation \eqref{4.5} with $y^{i_1}\dots y^{i_k}$, i.e., multiply the equation by $y^{i_1}\dots y^{i_k}$ and perform the summation over indices $i_1\dots i_k$
\begin{equation}
\begin{aligned}
\big(j_y^k\,&\widehat{N^{k}_mf}\big)_{i_{k+1}\dots  i_m}(y) =\frac{(-1)^{m\!+\!k}(2\pi)^{n/2}}{2^{2m\!+\!n/2\!-\!2}}\sum_{l=0}^k\binom{k}{l}
\frac{(-1)^l\Gamma\Big(\frac{1\!-\!2m\!-\!2l}{2}\Big)}{2^{2l}\Gamma\Big(\frac{2m\!+\!2l\!+\!n\!-\!1}{2}\Big)}\\
&\times\bigg[ y^{i_1}\dots y^{i_k}\partial_{i_1\cdots i_k}
\Big(\partial_{i_{k+1}\dots i_mj_1 \dots j_mp_1\dots p_{k+l}}|y|^{2m\!+\!2l\!-\!1}\Big)\bigg]
\partial^{p_1\dots p_{k+l}}{\widehat f}^{j_1\dots j_m}(y).
\end{aligned}
                           \label{4.6}
\end{equation}
On the right-hand side, all summands corresponding to $l>0$ are equal to zero. Indeed,\\
$\partial_{i_{k+1}\dots i_mj_1 \dots j_mp_1\dots p_{k+l}}|y|^{2m\!+\!2l\!-\!1}$ is the positively homogeneous function of degree $l-1$. By the Euler equation for homogeneous functions,
$$
\begin{aligned}
y^{i_1}\dots y^{i_k}\partial_{i_1\cdots i_k}
\Big(&\partial_{i_{k+1}\dots i_mj_1 \dots j_mp_1\dots p_{k+l}}|y|^{2m\!+\!2l\!-\!1}\Big)\\
&=\left\{\begin{array}{ll}(-1)^kk! \,\partial_{i_{k+1}\dots i_mj_1 \dots j_mp_1\dots p_{k+l}}|y|^{2m\!+\!2l\!-\!1}&\mbox{if}\ l=0,
\\0&\mbox{if}\ l>0.\end{array}\right.
\end{aligned}
$$
The formula \eqref{4.6} becomes
\begin{align}
    &\big(j_y^k\,\widehat{N^{k}_mf}\big)_{i_{k+1}\dots  i_m}\notag\\
&\qquad\qquad=\frac{(-1)^m(2\pi)^{n/2} k!\Gamma\big(\frac{1\!-\!2m}{2}\big)}{2^{2m\!+\!n/2\!-\!2}\Gamma\big(\frac{2m\!+\!n\!-\!1}{2}\big)}
\Big(\partial_{i_{k+1}\dots i_mj_1 \dots j_mp_1\dots p_k}|y|^{2m\!-\!1}\Big)
\partial^{p_1\dots p_k}{\widehat f}^{j_1\dots j_m}\notag\\
&\qquad\qquad=\frac{\pi^{(n+1)/2}\, k!}{2^{m-2}\,\Gamma\big(\frac{2m+n-1}{2}\big)}
\big(A^{m,0}/(d^k\widehat f)\big)_{i_{k+1}\dots  i_m}.\label{4.7}
\end{align}
This is equivalent to \eqref{3.7}.
\end{proof}

\begin{lemma} \label{L4.2}
Let $0\leq k\leq m$ and $n\ge2$. Then $j_y^{k+1}\,\widehat{N^{k}_mf}(y)=0$
for any tensor field $f\in \mathcal{S}\left(\mathbb{R}^n; S^m\right)$ and for any $y\in{\R}^n$.
\end{lemma}

\begin{proof}
The statement trivially holds in the case of $k=m$. In the case of $k<m$ we apply the operator $j_y$ to the first line of \eqref{4.7}
\begin{align*}
    &\big(j_y^{k+1}\,\widehat{N^{k}_mf}\big)_{i_{k+2}\dots  i_m}(y)
\\&\qquad\qquad=C_m^k\Big[\big(y^{i_{k+1}}\partial_{i_{k+1}}\big)
\Big(\partial_{i_{k+2}\dots i_mj_1 \dots j_mp_1\dots p_k}|y|^{2m\!-\!1}\Big)\Big]
\partial^{p_1\dots p_k}{\widehat f}^{j_1\dots j_m}(y),
\end{align*}
where $C_m^k=(-1)^m (2\pi)^{n/2}k!\Gamma\big(\frac{1\!-\!2m}{2}\big)/\Big(2^{2m\!+\!n/2\!-\!2}\Gamma\big(\frac{2m\!+\!n\!-\!1}{2}\big)\Big)$.
The expression in brackets is equal to zero since $\partial_{i_{k+2}\dots i_mj_1 \dots j_mp_1\dots p_k}|y|^{2m\!-\!1}$ is a positively homogeneous function of zero degree.
\end{proof}

From now on we can forget the momentum ray transform. The rest of the article is devoted to investigation of the system \eqref{3.7}.

Lemma \ref{L4.2} implies that right-hand sides of equations \eqref{3.7} satisfy
\begin{equation}
j_y\,F^{(m,k)}(y)=0\quad(0\le k\le m).
                           \label{4.8}
\end{equation}
Thus, equalities \eqref{4.8} constitute necessary conditions for existence of a solution $g\in {\mathcal S}({\R}^n;S^m)$ to the system \eqref{3.7}.
Equalities \eqref{4.8} appear to be necessary and sufficient consistency conditions for the system \eqref{3.7}, although this fact has not been formally proven.
\section{\texorpdfstring{Reduction of the system \eqref{3.7} to an algebraic system}{Reduction of the system (3.7) to an algebraic system}}

Tensor fields $A^{(m,k)}\in C^\infty({\R}^n\setminus\{0\};S^{2m-k})\ (0\le k\le m)$ are defined by \eqref{3.5}.
There exist two important relations between these tensor fields.

\begin{lemma} \label{L5.1}
The following equalities are valid:
\begin{equation}
j_yA^{(m,k)}=-kA^{(m-1,k-1)}\quad(0\le k\le m),
                           \label{5.1}
\end{equation}
\begin{equation}
\mbox{\rm div}\,A^{(m,k)}=(2m-2k-1)(n+2m-2k-3)A^{(m,k+1)}\quad(0\le k\le m-1).
                           \label{5.2}
\end{equation}
\end{lemma}

\begin{proof}
Applying the operator $j_y$ to the equality \eqref{3.5}, we have
$$
j_yA^{(m,k)}=j_yd^{2m-k}|y|^{2m-2k-1}.
$$
Using the operator $\langle y,\partial\rangle=y^j\frac{\partial}{\partial y^j}$, the latter formula can be written as
\begin{equation}
j_yA^{(m,k)}=\langle y,\partial\rangle\,d^{2m-k-1}|y|^{2m-2k-1}.
                           \label{5.3}
\end{equation}
The tensor field $d^{2m-k-1}|y|^{2m-2k-1}$ is positively homogeneous of degree $-k$. By the Euler equation for homogeneous functions,
$$
\langle y,\partial\rangle\,d^{2m-k-1}|y|^{2m-2k-1}=-k\,d^{2m-k-1}|y|^{2m-2k-1}.
$$
Substituting this expression into \eqref{5.3}, we obtain
$$
j_yA^{(m,k)}=-k\,d^{2m-k-1}|y|^{2m-2k-1}.
$$
By \eqref{3.5}, the right-hand side of this formula is equal to $-kA^{(m-1,k-1)}$. This proves \eqref{5.1}.

Let us write \eqref{3.5} in the coordinate form
$$
A^{(m,k)}_{i_{k+1}\dots i_{2m}}=\partial_{i_{k+1}\dots i_{2m}}|y|^{2m-2k-1}.
$$
Differentiate this equality
$$
\frac{\partial A^{(m,k)}_{i_{k+1}\dots i_{2m}}}{\partial y^j}
=\partial_{ji_{k+1}\dots i_{2m}}|y|^{2m-2k-1}.
$$
From this
$$
\begin{aligned}
(\mbox{\rm div}\,A^{(m,k)})_{i_{k+2}\dots i_{2m}}
&=\delta^{jl}\,\frac{\partial A^{(m,k)}_{li_{k+2}\dots i_{2m}}}{\partial y^j}
=\delta^{jl}\,\partial_{jli_{k+2}\dots i_{2m}}|y|^{2m-2k-1}\\
&=\partial_{i_{k+2}\dots i_{2m}}(\Delta|y|^{2m-2k-1}).
\end{aligned}
$$
This can be written in the coordinate-free form
\begin{equation}
\mbox{\rm div}\,A^{(m,k)}=d^{2m-k-1}(\Delta|y|^{2m-2k-1}).
                           \label{5.4}
\end{equation}
Using the well-known formula
$$
\Delta|y|^\alpha=\alpha(\alpha+n-2)|y|^{\alpha-2},
$$
we obtain
$$
\Delta|y|^{2m-2k-1}=(2m-2k-1)(n+2m-2k-3)|y|^{2m-2k-3}.
$$
Substituting this expression into \eqref{5.4}, we have
$$
\mbox{\rm div}\,A^{(m,k)}=(2m-2k-1)(n+2m-2k-3)\,d^{2m-k-1}|y|^{2m-2k-3}.
$$
By \eqref{3.5},
$$
d^{2m-k-1}|y|^{2m-2k-3}=A^{(m,k+1)}.
$$
The two last formulas imply \eqref{5.2}.
\end{proof}

From \eqref{5.2}, one can easily prove by induction on $k$
\begin{equation}
\mbox{\rm div}^k\,A^{(m,0)}=\frac{(2m-1)!!(n+2m-3)!!}{(2m-2k-1)!!(n+2m-2k-3)!!}\,A^{(m,k)}\quad(0\le k\le m).
                           \label{5.5}
\end{equation}

We reproduce the system \eqref{3.7}
\begin{equation}
A^{(m,0)}/(d^lg)=F^{(m,l)}\quad(l=0,1,\dots,m).
                           \label{5.6}
\end{equation}
Here $F^{(m,l)}\in C^\infty({\R}^n\setminus\{0\};S^{m-l})\ (0\le l\le m)$ are arbitrary tensor fields belonging to the kernel of $j_y$.

\begin{proposition} \label{P5.1}
If a tensor field $g\in C^\infty({\R}^n\setminus\{0\};S^m)$ satisfies \eqref{5.6}, then
\begin{align}
    \begin{split}
       & (\mbox{\rm div}^kA^{(m,0)})/(d^lg)\\&\qquad\qquad=(-1)^k\sum\limits_{p=0}^k(-1)^p\binom{k}{p}\,\mbox{\rm div}^p\,F^{(m,k+l-p)}
\quad(0\le k\le m, 0\le l\le m-k).
                           \label{5.7}
    \end{split}
\end{align}
\end{proposition}

\begin{proof}
We prove \eqref{5.7} by induction on $k$. For $k=0$, \eqref{5.7} coincides with \eqref{5.6}. Assume \eqref{5.7} to be valid for some $k$. Apply the operator $\mbox{\rm div}$ to the equation \eqref{5.7}
\begin{equation}
\mbox{\rm div}\Big((\mbox{\rm div}^kA^{(m,0)})/(d^lg)\Big)
=(-1)^k\sum\limits_{p}(-1)^p\binom{k}{p}\,\mbox{\rm div}^{p+1}\,F^{(m,k+l-p)}.
                           \label{5.8}
\end{equation}
We assume binomial coefficients to be defined for all integers $k$ and $p$ with the convention:
\begin{equation}
\binom{k}{p}=0\quad \mbox{if either}\  k<0\ \mbox{or}\ p<0\ \mbox{or}\ k<p.
                           \label{5.8a}
\end{equation}
 With this convention, we can assume the summation to be performed over all integers $p$ in \eqref{5.8} and the formulas below.

The equality
$$
\mbox{\rm div}(u/v)=(\mbox{\rm div}\,u)/v+u/(dv)
$$
is valid for any two tensor fields $u$ and $v$. This can be readily proved based on the definitions of the operators $d$ and $\rm div$. Using this equality, we write \eqref{5.8} in the form
$$
(\mbox{\rm div}^{k+1}A^{(m,0)})/(d^lg)+(\mbox{\rm div}^kA^{(m,0)})/(d^{l+1}g)
=(-1)^k\sum\limits_{p}(-1)^p\binom{k}{p}\,\mbox{\rm div}^{p+1}\,F^{(m,k+l-p)}.
$$
By the induction hypothesis,
$$
(\mbox{\rm div}^kA^{(m,0)})/(d^{l+1}g)
=(-1)^k\sum\limits_{p}(-1)^p\binom{k}{p}\,\mbox{\rm div}^p\,F^{(m,k+l-p+1)}.
$$
Substituting this expression into the previous formula, we write the result in the form
$$
\begin{aligned}
(\mbox{\rm div}^{k+1}A^{(m,0)})/(d^lg)
&=(-1)^k\sum\limits_{p}(-1)^p\binom{k}{p}\,\mbox{\rm div}^{p+1}\,F^{(m,k+l-p)}\\
&+(-1)^{k+1}\sum\limits_{p}(-1)^p\binom{k}{p}\,\mbox{\rm div}^p\,F^{(m,k+l-p+1)}.
\end{aligned}
$$
Changing the summation variable of the first sum as $p:=p-1$, we obtain
$$
(\text{div}^{k+1}A^{(m,0)})/(d^lg)
=(-1)^{k+1}\sum\limits_{p}(-1)^p\bigg[\binom{k}{p-1}+\binom{k}{p}\bigg]\text{div}^p\,F^{(m,k+l-p+1)}.
$$
By the Pascal triangle equality, $\binom{k}{p-1}+\binom{k}{p}=\binom{k+1}{p}$. Substituting this expression into the last formula, we arrive at \eqref{5.7} for $k:=k+1$.
\end{proof}

Setting $l=0$ in \eqref{5.7}, we obtain
$$
(\mbox{\rm div}^kA^{(m,0)})/g=(-1)^k\sum\limits_{p=0}^k(-1)^p\binom{k}{p}\,\mbox{\rm div}^p\,F^{(m,k-p)}
\quad(0\le k\le m).
$$
Substituting the value \eqref{5.5} of $\mbox{\rm div}^kA^{(m,0)}$, we arrive at the equation
$$
A^{(m,k)}/g=(-1)^k\frac{(2m-2k-1)!!(n+2m-2k-3)!!}{(2m-1)!!(n+2m-3)!!}
\sum\limits_{p=0}^k(-1)^p\binom{k}{p}\,\mbox{\rm div}^p\,F^{(m,k-p)}.
$$
We have thus proved

\begin{proposition} \label{P5.2}
If a tensor field $g\in C^\infty({\R}^n\setminus\{0\};S^m)$ solves the system \eqref{3.7}, then it also solves the system \eqref{3.8} with right-hand sides defined by \eqref{3.9}.
\end{proposition}

We emphasize that \eqref{3.8} is a system of linear algebraic equations in coordinates of the unknown tensor field $g$. Of course the system \eqref{3.8} is not equivalent to \eqref{3.7}. Proposition \ref{P5.2} states that \eqref{3.7} implies \eqref{3.8} but not vice versa. Nevertheless, we will see that $g$ can be uniquely recovered from \eqref{3.8}. In this sense the system \eqref{3.7} is reduced to the algebraic system \eqref{3.8}.

\section{\texorpdfstring{Solution of the system \eqref{3.8}}{Solution of the system (3.8)}}

The following statement actually contains the most important part of the proof of Theorem \ref{Th3.1}.

\begin{proposition} \label{P6.1}
If the system \eqref{3.8} is solvable, then the solution is unique and is expressed by the formula
\begin{equation}
g(y)=\frac{|y|}{m!}\sum\limits_{k=0}^m\frac{(-1)^k}{(2m\!-\!2k\!-\!1)!!} \binom{m}{k}
\sum\limits_{p=0}^{\min(k,m-k)}\!\!\frac{(-1)^p}{2^p}\binom{m - k}{p}\, i^{p}i_{y}^{k-p} j_{y}^{p}H^{(m,k)}(y).
                           \label{6.0}
\end{equation}
\end{proposition}

The proof of Proposition \ref{P6.1} involves several steps. The main part of the proof is contained in the following two lemmas.

\begin{lemma} \label{L6.1}
Given a tensor field  $g\in C^\infty({\R}^n\setminus\{0\};S^{m+1})$, let us fix a Cartesian coordinate system on ${\R}^n$, fix a value of the index $i_{m+1}$ and introduce the tensor field
$\tilde g\in C^\infty({\R}^n\setminus\{0\};S^m)$ by
\begin{equation}
\tilde g_{i_1\dots i_m}=g_{i_1\dots i_mi_{m+1}}.
                           \label{6.1}
\end{equation}
Let us also introduce the vector field $\tilde\delta$ by
$$
\tilde\delta_i=\delta_{ii_{m+1}}.
$$
Then, for every $0\le k\le m$,
\begin{align}\label{6.2}
    \begin{split}
       & (A^{(m,k)}/\tilde g)_{i_{k+1}\dots i_m}
=\frac{1}{m\!+\!1}\Big[\frac{1}{2m\!-\!2k\!+\!1}\,(A^{(m+1,k)}/g)_{i_{k+1}\dots i_{m+1}}\\
&\qquad\qquad-y_{i_{m+1}}\,(A^{(m+1,k+1)}/g)_{i_{k+1}\dots i_m}
-(m\!-\!k)\,\Big(\tilde\delta\,(A^{(m,k)}/g)\Big)_{i_{k+1}\dots i_m}\Big].
    \end{split}
\end{align}
\end{lemma}

\begin{proof}
%
The identity
$$
\partial_{i_1\dots i_p}(|y|^\alpha y_k)
=y_k\,\partial_{i_1\dots i_p}|y|^\alpha
+p\,\sigma(i_1\dots i_p)(\delta_{i_1k}\,\partial_{i_2\dots i_p}|y|^\alpha)
$$
holds for any integer $p\ge0$, any real $\alpha$ and any $1\le k\le n$. It is easily proved by induction on $p$.
Using this identity, we obtain
$$
\begin{aligned}
&\partial_{i_{k+1}\dots i_{m+1}j_1\dots j_{m+1}}|y|^{2m-2k+1}
=(2m\!-\!2k\!+\!1)\partial_{i_{k+1}\dots i_m j_1\dots j_{m+1}}(|y|^{2m-2k-1}y_{i_{m+1}})\\
&=(2m\!-\!2k\!+\!1)\bigg[y_{i_{m+1}}\,\partial_{i_{k+1}\dots i_m j_1\dots j_{m+1}}|y|^{2m-2k-1}\\
&+(2m\!-\!k\!+\!1)\sigma(i_{k+1}\dots i_m j_1\dots j_{m+1})\Big(\delta_{i_{m+1}j_{m+1}}\,\partial_{i_{k+1}\dots i_m j_1\dots j_m}|y|^{2m-2k-1}\Big)\bigg].
\end{aligned}
$$
Expanding the symmetrization $\sigma(i_{k+1}\dots i_m j_1\dots j_{m+1})$ with respect to the index $j_{m+1}$
(see \cite[Lemma 2.4.1]{sharafutdinov2012integral}), we write this in the form
$$
\begin{aligned}
\frac{1}{2m\!-\!2k\!+\!1}\,&\partial_{i_{k+1}\dots i_{m+1}j_1\dots j_{m+1}}|y|^{2m-2k+1}
=y_{i_{m+1}}\,\partial_{i_{k+1}\dots i_m j_1\dots j_{m+1}}|y|^{2m-2k-1}\\
&+\sigma(i_{k+1}\dots i_m j_1\dots j_m)\Big(\delta_{i_{m+1}j_{m+1}}\,\partial_{i_{k+1}\dots i_m j_1\dots j_m}|y|^{2m-2k-1}\\
&+(2m\!-\!k)\delta_{i_{m+1}j_1}\,\partial_{i_{k+1}\dots i_m j_2\dots j_{m+1}}|y|^{2m-2k-1}\Big).
\end{aligned}
$$
Since the tensor $\delta_{i_{m+1}j_{m+1}}\,\partial_{i_{k+1}\dots i_m j_1\dots j_m}|y|^{2m-2k-1}$ is symmetric in
$i_{k+1},\dots, i_m, j_1,\dots, j_m$, the formula can be written as follows:
\begin{equation}
\begin{aligned}
&\frac{1}{2m\!-\!2k\!+\!1}\,\partial_{i_{k+1}\dots i_{m+1}j_1\dots j_{m+1}}|y|^{2m-2k+1}
=y_{i_{m+1}}\,\partial_{i_{k+1}\dots i_m j_1\dots j_{m+1}}|y|^{2m-2k-1}\\
&+\delta_{i_{m+1}j_{m+1}}\,\partial_{i_{k+1}\dots i_m j_1\dots j_m}|y|^{2m-2k-1}\\
&+(2m\!-\!k)\sigma(i_{k+1}\dots i_m j_1\dots j_m)\Big(\delta_{i_{m+1}j_1}\,\partial_{i_{k+1}\dots i_m j_2\dots j_{m+1}}|y|^{2m-2k-1}\Big).
\end{aligned}
                           \label{6.3}
\end{equation}

By \eqref{3.5},
$$
\begin{aligned}
\partial_{i_{k+1}\dots i_m j_1\dots j_{m+1}}|y|^{2m-2k-1}&=A^{(m+1,k+1)}_{i_{k+1}\dots i_m j_1\dots j_{m+1}},\\
\partial_{i_{k+1}\dots i_m j_1\dots j_m}|y|^{2m-2k-1}&=A^{(m,k)}_{i_{k+1}\dots i_m j_1\dots j_m},\\
\partial_{i_{k+1}\dots i_m j_2\dots j_{m+1}}|y|^{2m-2k-1}&=A^{(m,k)}_{i_{k+1}\dots i_m j_2\dots j_{m+1}}.
\end{aligned}
$$
Substituting these expressions into \eqref{6.3}
\begin{equation}
\begin{aligned}
&\frac{1}{2m\!-\!2k\!+\!1}\,\partial_{i_{k+1}\dots i_{m+1}j_1\dots j_{m+1}}|y|^{2m-2k+1}
=y_{i_{m+1}}\,A^{(m+1,k+1)}_{i_{k+1}\dots i_m j_1\dots j_{m+1}}\\
&+\delta_{i_{m+1}j_{m+1}}\,A^{(m,k)}_{i_{k+1}\dots i_m j_1\dots j_m}\\
&+(2m\!-\!k)\,\sigma(i_{k+1}\dots i_m j_1\dots j_m)\Big(\delta_{i_{m+1}j_1}\,A^{(m,k)}_{i_{k+1}\dots i_m j_2\dots j_{m+1}}\Big).
\end{aligned}
                           \label{6.4}
\end{equation}

The equality
\begin{equation}
\begin{aligned}
&\sigma(i_{k+1}\dots i_m j_1\dots j_m)\Big(\delta_{i_{m+1}j_1}\,A^{(m,k)}_{i_{k+1}\dots i_m j_2\dots j_{m+1}}\Big)\\
&=\frac{1}{2m\!-\!k}\sigma(i_{k+1}\dots i_m)\sigma(j_1\dots j_m)
\Big[(m\!-\!k)\delta_{i_{m+1}i_m}\,A^{(m,k)}_{i_{k+1}\dots i_{m-1} j_1\dots j_{m+1}}\\
&\qquad\qquad\qquad\qquad\qquad\qquad\qquad\qquad+m\delta_{i_{m+1}j_1}\,A^{(m,k)}_{i_{k+1}\dots i_m j_2\dots j_{m+1}}\Big]
\end{aligned}
                           \label{6.5}
\end{equation}
can be easily proved based on the fact that the tensor $A^{(m,k)}$ is symmetric. Formally speaking, the first term
$(m\!-\!k)\delta_{i_{m+1}i_m}\,A^{(m,k)}_{i_{k+1}\dots i_{m-1} j_1\dots j_{m+1}}$ in brackets makes sense for $k\le m-2$ only. Nevertheless, the formula \eqref{6.5} holds for $k=m-1$ if we assume that
$A^{(m,k)}_{i_m\dots i_{m-1} j_1\dots j_{m+1}}=A^{(m,k)}_{j_1\dots j_{m+1}}$.
In the case of $k=m$, the first term in brackets is equal to zero because of the factor $(m-k)$. Thus, the formula \eqref{6.5} holds for all $0\le k\le m$.

Using \eqref{6.5}, the formula \eqref{6.4} becomes
$$
\begin{aligned}
&\frac{1}{2m\!-\!2k\!+\!1}\,\partial_{i_{k+1}\dots i_{m+1}j_1\dots j_{m+1}}|y|^{2m-2k+1}
=y_{i_{m+1}}\,A^{(m+1,k+1)}_{i_{k+1}\dots i_m j_1\dots j_{m+1}}\\
&+\delta_{i_{m+1}j_{m+1}}\,A^{(m,k)}_{i_{k+1}\dots i_m j_1\dots j_m}\\
&+(m\!-\!k)\,\sigma(i_{k+1}\dots i_m)\sigma(j_1\dots j_m)\Big(\delta_{i_{m+1}i_m}\,A^{(m,k)}_{i_{k+1}\dots i_{m-1} j_1\dots j_{m+1}}\Big)\\
&+m\,\sigma(i_{k+1}\dots i_m)\sigma(j_1\dots j_m)\Big(\delta_{i_{m+1}j_1}\,A^{(m,k)}_{i_{k+1}\dots i_m j_2\dots j_{m+1}}\Big).
\end{aligned}
$$
In the last line, the symmetrization$\sigma(i_{k+1}\dots i_m)$ is unnecessary, as the tensor
$A^{(m,k)}_{i_{k+1}\dots i_m j_2\dots j_{m+1}}$ is already symmetric in these indices. The formula simplifies to the following one:
\begin{equation}
\begin{aligned}
&\frac{1}{2m\!-\!2k\!+\!1}\,\partial_{i_{k+1}\dots i_{m+1}j_1\dots j_{m+1}}|y|^{2m-2k+1}
=y_{i_{m+1}}\,A^{(m+1,k+1)}_{i_{k+1}\dots i_m j_1\dots j_{m+1}}\\
&+\delta_{i_{m+1}j_{m+1}}\,A^{(m,k)}_{i_{k+1}\dots i_m j_1\dots j_m}\\
&+(m\!-\!k)\,\sigma(i_{k+1}\dots i_m)\sigma(j_1\dots j_m)\Big(\delta_{i_{m+1}i_m}\,A^{(m,k)}_{i_{k+1}\dots i_{m-1} j_1\dots j_{m+1}}\Big)\\
&+m\,\sigma(j_1\dots j_m)\Big(\delta_{i_{m+1}j_1}\,A^{(m,k)}_{i_{k+1}\dots i_m j_2\dots j_{m+1}}\Big).
\end{aligned}
                           \label{6.6}
\end{equation}

The formulas \eqref{6.3} and \eqref{6.6} imply
$$
\begin{aligned}
&\frac{1}{2m\!-\!2k\!+\!1}\,(A^{(m+1,k)}/g)_{i_{k+1}\dots i_{m+1}}
=\bigg[y_{i_{m+1}}\,A^{(m+1,k+1)}_{i_{k+1}\dots i_m j_1\dots j_{m+1}}\\
&+\delta_{i_{m+1}j_{m+1}}\,A^{(m,k)}_{i_{k+1}\dots i_m j_1\dots j_m}\\
&+(m\!-\!k)\,\sigma(i_{k+1}\dots i_m)\sigma(j_1\dots j_m)\Big(\delta_{i_{m+1}i_m}\,A^{(m,k)}_{i_{k+1}\dots i_{m-1} j_1\dots j_{m+1}}\Big)\\
&+m\,\sigma(j_1\dots j_m)\Big(\delta_{i_{m+1}j_1}\,A^{(m,k)}_{i_{k+1}\dots i_m j_2\dots j_{m+1}}\Big)\bigg]
g^{j_1\dots j_{m+1}}.
\end{aligned}
$$
The symmetrization $\sigma(j_1\dots j_m)$ can be omitted after expanding the expression, as the tensor $g^{j_1\dots j_{m+1}}$ is symmetric in these indices. We thus obtain
$$
\begin{aligned}
\frac{1}{2m\!-\!2k\!+\!1}\,&(A^{(m+1,k)}/g)_{i_{k+1}\dots i_{m+1}}
=y_{i_{m+1}}\,A^{(m+1,k+1)}_{i_{k+1}\dots i_m j_1\dots j_{m+1}}g^{j_1\dots j_{m+1}}\\
&+\delta_{i_{m+1}j_{m+1}}\,A^{(m,k)}_{i_{k+1}\dots i_m j_1\dots j_m}g^{j_1\dots j_{m+1}}\\
&+(m\!-\!k)\,\sigma(i_{k+1}\dots i_m)\Big(\delta_{i_{m+1}i_m}\,A^{(m,k)}_{i_{k+1}\dots i_{m-1} j_1\dots j_{m+1}}g^{j_1\dots j_{m+1}}\Big)\\
&+m\,\delta_{i_{m+1}j_1}\,A^{(m,k)}_{i_{k+1}\dots i_m j_2\dots j_{m+1}}g^{j_1\dots j_{m+1}}.
\end{aligned}
$$
Implementing the contraction with the Kronecker tensor in second and last lines, we obtain
$$
\begin{aligned}
\frac{1}{2m\!-\!2k\!+\!1}\,&(A^{(m+1,k)}/g)_{i_{k+1}\dots i_{m+1}}
=y_{i_{m+1}}\,A^{(m+1,k+1)}_{i_{k+1}\dots i_m j_1\dots j_{m+1}}g^{j_1\dots j_{m+1}}\\
&+A^{(m,k)}_{i_{k+1}\dots i_m j_1\dots j_m}g_{i_{m+1}}^{j_1\dots j_m}\\
&+(m\!-\!k)\,\sigma(i_{k+1}\dots i_m)\Big(\delta_{i_{m+1}i_m}\,A^{(m,k)}_{i_{k+1}\dots i_{m-1} j_1\dots j_{m+1}}g^{j_1\dots j_{m+1}}\Big)\\
&+m\,A^{(m,k)}_{i_{k+1}\dots i_m j_1\dots j_m}g_{i_{m+1}}^{j_1\dots j_m}.
\end{aligned}
$$
In the last line, we have replaced the summation indices $j_2,\dots, j_{m+1}$ with $j_1,\dots, j_m$. We see now that second and last lines contain similar terms. Grouping these terms, we write the formula as follows:
\begin{equation}
\begin{aligned}
\frac{1}{2m\!-\!2k\!+\!1}\,&(A^{(m+1,k)}/g)_{i_{k+1}\dots i_{m+1}}
=y_{i_{m+1}}\,A^{(m+1,k+1)}_{i_{k+1}\dots i_m j_1\dots j_{m+1}}g^{j_1\dots j_{m+1}}\\
&+(m\!-\!k)\,\sigma(i_{k+1}\dots i_m)\Big(\delta_{i_{m+1}i_m}\,A^{(m,k)}_{i_{k+1}\dots i_{m-1} j_1\dots j_{m+1}}g^{j_1\dots j_{m+1}}\Big)\\
&+(m+1)\,A^{(m,k)}_{i_{k+1}\dots i_m j_1\dots j_m}g_{i_{m+1}}^{j_1\dots j_m}.
\end{aligned}
                           \label{6.7}
\end{equation}

Recall that the value of the index $i_{m+1}$ is fixed and $\tilde g^{j_1\dots j_m}=g_{i_{m+1}}^{j_1\dots j_m}$.
The formula \eqref{6.7} can be written as
$$
\begin{aligned}
\frac{1}{2m\!-\!2k\!+\!1}\,&(A^{(m+1,k)}/g)_{i_{k+1}\dots i_{m+1}}
=y_{i_{m+1}}\,(A^{(m+1,k+1)}/g)_{i_{k+1}\dots i_m}\\
&+(m\!-\!k)\,\Big(\tilde\delta\,(A^{(m,k)}/g)\Big)_{i_{k+1}\dots i_m}
+(m+1)\,(A^{(m,k)}/\tilde g)_{i_{k+1}\dots i_m}.
\end{aligned}
$$
This is equivalent to \eqref{6.2}.
\end{proof}

\begin{lemma} \label{L6.2}
For a tensor field $g\in C^\infty({\R}^n\setminus\{0\};S^m)$, the following identity holds:
\begin{equation}
A^{(m,0)}/g=m!(2m-1)!!\,|y|^{-1}g
+\sum\limits_{k=1}^m
\sum\limits_{p=0}^{\min(k,m-k)}\!\!\beta(m,k,p)\,i_y^{k-p}i^pj_y^p(A^{(m,k)}/g),
                           \label{6.8}
\end{equation}
where the coefficients $\beta(m,k,p)$ are uniquely determined by the recurrent formulas
\begin{equation}
\begin{aligned}
\tilde\beta(m\!+\!1,k,p)
&=\frac{1}{2m\!-\!2k\!+\!1}\,\beta(m,k,p)
-\frac{k\!-\!p}{k}\,\beta(m,k\!-\!1,p)\\
&+\frac{m\!-\!k\!-\!p\!+\!2}{\textstyle k}\,\beta(m,k\!-\!1,p\!-\!1)
\end{aligned}
                           \label{6.9}
\end{equation}
and
\begin{equation}
\beta(m+1,k,p)=\left\{\begin{array}{ll}
(2m\!+\!1)\big(\tilde\beta(m+1,k,p)+1\big)&\mbox{\rm if}\ (k,p)=(1,0),\\
[8pt]
(2m\!+\!1)\big(\tilde\beta(m+1,k,p)-m\big)&\mbox{\rm if}\ (k,p)=(1,1),\\
[8pt]
(2m\!+\!1)\tilde\beta(m+1,k,p)&\mbox{\rm otherwise}
\end{array}\right.
                           \label{6.10}
\end{equation}
with the convention
\begin{equation}
\beta(m,k,p)=0\quad\mbox{\rm if either}\ k=0\ \mbox{\rm or}\ k>m\ \mbox{\rm or}\ p<0\ \mbox{\rm or}\ p>\min(k,m-k).
                           \label{6.11}
\end{equation}
\end{lemma}

\begin{proof}
The proof proceeds by induction on $m$. For $m=0$, the sum on the right-hand side of \eqref{6.8} is absent and the formula holds since $A^{(m,0)}=|y|^{-1}$. Assume \eqref{6.8} to be valid for some $m\ge0$ and let $g\in C^\infty({\R}^n\setminus\{0\};S^{m+1})$. We fix a value of the index $i_{m+1}$ and introduce the tensor field $\tilde g\in C^\infty({\R}^n\setminus\{0\};S^m)$ by \eqref{6.1}. By the induction hypothesis, the formula \eqref{6.8} holds for $\tilde g$. Let us write the formula in coordinates
\begin{equation}
\begin{aligned}
&A^{(m,0)}_{i_1\dots i_mj_1\dots j_m}g_{i_{m+1}}^{j_1\dots j_m}=m!(2m-1)!!\,|y|^{-1}g_{i_1\dots i_{m+1}}\\
&+\sigma(i_1\dots i_m)\sum\limits_{k=1}^m
\sum\limits_{p=0}^{\min(k,m-k)}\!\!\beta(m,k,p)\,
y^{l_1}\dots y^{l_p}\times\\
&\qquad\qquad\qquad\times\delta_{i_1i_2}\dots\delta_{i_{2p-1}i_{2p}}y_{i_{2p+1}}\dots y_{i_{k+p}}
A^{(m,k)}_{i_{k+p+1}\dots i_ml_1\dots l_pj_1\dots j_m}g_{i_{m+1}}^{j_1\dots j_m}.
\end{aligned}
                           \label{6.12}
\end{equation}
By Lemma \ref{L6.1},
\begin{equation}
\begin{aligned}
&A^{(m,k)}_{i_{k+p+1}\dots i_ml_1\dots l_pj_1\dots j_m}g_{i_{m+1}}^{j_1\dots j_m}
=(A^{(m,k)}/\tilde g)_{i_{k+p+1}\dots i_ml_1\dots l_p}\\
&=\frac{1}{m\!+\!1}\Big[\frac{1}{2m\!-\!2k\!+\!1}\,(A^{(m+1,k)}/g)_{i_{k+p+1}\dots i_{m+1}l_1\dots l_p}
\\
&\quad-y_{i_{m+1}}\,(A^{(m+1,k+1)}/g)_{i_{k+p+1}\dots i_ml_1\dots l_p}-(m\!-\!k)\,\Big(\tilde\delta\,(A^{(m,k)}/g)\Big)_{i_{k+p+1}\dots i_ml_1\dots l_p}\Big].
\end{aligned}
                           \label{6.13}
\end{equation}

In the case of $k=m$, the last term on the right-hand side of \eqref{6.13} is equal to zero. In the case of $k<m$, we transform the last term on the right-hand side of \eqref{6.13} using \eqref{5.1} as follows:
\begin{equation}
\begin{aligned}
&\Big(\tilde\delta\,(A^{(m,k)}/g)\Big)_{i_{k+p+1}\dots i_ml_1\dots l_p}
=-\frac{1}{k+1}\Big(\tilde\delta(j_yA^{(m+1,k+1)}/g)\Big)_{i_{k+p+1}\dots i_ml_1\dots l_p}\\
&=-\frac{1}{k+1}\sigma(i_{k+p+1}\dots i_ml_1\dots l_p)
\Big(\delta_{i_{m+1}i_{k+p+1}}(j_yA^{(m+1,k+1)}/g)_{i_{k+p+2}\dots i_ml_1\dots l_p}\Big).
\end{aligned}
                           \label{6.14}
\end{equation}
The equality
$$
\begin{aligned}
\sigma(i_{k+p+1}\dots i_m&l_1\dots l_p)
\Big(\delta_{i_{m+1}i_{k+p+1}}(j_yA^{(m+1,k+1)}/g)_{i_{k+p+2}\dots i_ml_1\dots l_p}\Big)\\
&=\frac{m\!-\!k\!-\!p}{m\!-\!k}\,\sigma(i_{k+p+1}\dots i_m)
\Big(\delta_{i_{m+1}i_{k+p+1}}(j_yA^{(m+1,k+1)}/g)_{i_{k+p+2}\dots i_ml_1\dots l_p}\Big)\\
&+\frac{p}{m\!-\!k}\,\sigma(l_1\dots l_p)
\Big(\delta_{i_{m+1}l_1}(j_yA^{(m+1,k+1)}/g)_{i_{k+p+1}\dots i_ml_2\dots l_p}\Big)
\end{aligned}
$$
holds since $j_yA^{(m+1,k+1)}/g$ is a symmetric tensor.
Using this, the formula \eqref{6.14} takes the form
\begin{equation}
\begin{aligned}
&\Big(\tilde\delta\,(A^{(m,k)}/g)\Big)_{i_{k+p+1}\dots i_ml_1\dots l_p}\\
&=-\frac{m\!-\!k\!-\!p}{(m\!-\!k)(k\!+\!1)}\,\sigma(i_{k+p+1}\dots i_m)
\Big(\delta_{i_{m+1}i_{k+p+1}}(j_yA^{(m+1,k+1)}/g)_{i_{k+p+2}\dots i_ml_1\dots l_p}\Big)\\
&-\frac{p}{(m\!-\!k)(k\!+\!1)}\,\sigma(l_1\dots l_p)
\Big(\delta_{i_{m+1}l_1}(j_yA^{(m+1,k+1)}/g)_{i_{k+p+1}\dots i_ml_2\dots l_p}\Big).
\end{aligned}
                           \label{6.15}
\end{equation}

Replacing the last term on the right-hand side of \eqref{6.13} with its value \eqref{6.15}, we obtain
\begin{equation}
\begin{aligned}
&A^{(m,k)}_{i_{k+p+1}\dots i_ml_1\dots l_pj_1\dots j_m}g_{i_{m+1}}^{j_1\dots j_m}
=(A^{(m,k)}/\tilde g)_{i_{k+p+1}\dots i_ml_1\dots l_p}\\
&=\frac{1}{m\!+\!1}\bigg[\frac{1}{2m\!-\!2k\!+\!1}\,(A^{(m+1,k)}/g)_{i_{k+p+1}\dots i_{m+1}l_1\dots l_p}
-y_{i_{m+1}}\,(A^{(m+1,k+1)}/g)_{i_{k+p+1}\dots i_ml_1\dots l_p}\\
&\qquad\qquad+\frac{m\!-\!k\!-\!p}{k\!+\!1}\,\sigma(i_{k+p+1}\dots i_m)
\Big(\delta_{i_{m+1}i_{k+p+1}}(j_yA^{(m+1,k+1)}/g)_{i_{k+p+2}\dots i_ml_1\dots l_p}\Big)\\
&\qquad\qquad+\frac{p}{k\!+\!1}\,\sigma(l_1\dots l_p)
\Big(\delta_{i_{m+1}l_1}(j_yA^{(m+1,k+1)}/g)_{i_{k+p+1}\dots i_ml_2\dots l_p}\Big)
\bigg].
\end{aligned}
                           \label{6.16}
\end{equation}
It is not quite obvious now that two last lines on the right-hand side of \eqref{6.16} are equal to zero in the case of $k=m$. Nevertheless, in the case of $k=m$ we are interested in \eqref{6.16} for $p=0$ only, as is seen from \eqref{6.12}. For $k=m$ and $p=0$,  two last lines on the right-hand side of \eqref{6.16} are equal to zero.

Substituting the expression \eqref{6.16} into \eqref{6.12}. After the substitution, the symmetrization $\sigma(i_{k+p+1}\dots i_m)$ can be omitted because of the presence of the ``larger'' symmetrization $\sigma(i_1\dots i_m)$. The symmetrization $\sigma(l_1\dots l_p)$ can be also omitted because of the presence of the factor $y^{l_1}\dots y^{l_p}$. We thus obtain
$$
\begin{aligned}
&A^{(m,0)}_{i_1\dots i_mj_1\dots j_m}g_{i_{m+1}}^{j_1\dots j_m}=m!(2m-1)!!\,|y|^{-1}g_{i_1\dots i_{m+1}}\\
&+\frac{1}{m\!+\!1}\,\sigma(i_1\dots i_m)\sum\limits_{k=1}^m
\sum\limits_{p=0}^{\min(k,m-k)}\!\!\beta(m,k,p)\,
\delta_{i_1i_2}\dots\delta_{i_{2p-1}i_{2p}}\,y_{i_{2p+1}}\dots y_{i_{k+p}}\,y^{l_1}\dots y^{l_p}\times\\
&\times
\bigg[\frac{1}{2m\!-\!2k\!+\!1}\,(A^{(m+1,k)}/g)_{i_{k+p+1}\dots i_{m+1}l_1\dots l_p}
-y_{i_{m+1}}\,(A^{(m+1,k+1)}/g)_{i_{k+p+1}\dots i_ml_1\dots l_p}\\
&\qquad+\frac{m\!-\!k\!-\!p}{k\!+\!1}\,\delta_{i_{m+1}i_{k+p+1}}(j_yA^{(m+1,k+1)}/g)_{i_{k+p+2}\dots i_ml_1\dots l_p}\\
&\qquad+\frac{p}{k\!+\!1}\,\delta_{i_{m+1}l_1}(j_yA^{(m+1,k+1)}/g)_{i_{k+p+1}\dots i_ml_2\dots l_p}
\bigg].
\end{aligned}
$$
After pulling the factor $y^{l_1}\dots y^{l_p}$ inside brackets, this becomes
$$
\begin{aligned}
&A^{(m,0)}_{i_1\dots i_mj_1\dots j_m}g_{i_{m+1}}^{j_1\dots j_m}=m!(2m-1)!!\,|y|^{-1}g_{i_1\dots i_{m+1}}\\
&+\frac{1}{m\!+\!1}\,\sigma(i_1\dots i_m)\sum\limits_{k=1}^m
\sum\limits_{p=0}^{\min(k,m-k)}\!\!\beta(m,k,p)\,
\delta_{i_1i_2}\dots\delta_{i_{2p-1}i_{2p}}\,y_{i_{2p+1}}\dots y_{i_{k+p}}\times\\
&\times
\bigg[\frac{1}{2m\!-\!2k\!+\!1}\,(j_y^p\,A^{(m+1,k)}/g)_{i_{k+p+1}\dots i_{m+1}}
-y_{i_{m+1}}\,(j_y^p\,A^{(m+1,k+1)}/g)_{i_{k+p+1}\dots i_m}\\
&\qquad+\frac{m\!-\!k\!-\!p}{k\!+\!1}\,\delta_{i_{m+1}i_{k+p+1}}(j_y^{p+1}\,A^{(m+1,k+1)}/g)_{i_{k+p+2}\dots i_m}\\
&\qquad+\frac{p}{k\!+\!1}\,y_{i_{m+1}}(j_y^p\,A^{(m+1,k+1)}/g)_{i_{k+p+1}\dots i_m}
\bigg].
\end{aligned}
$$
Observe that second and last terms in brackets differ by coefficients only. After grouping these terms, the formula becomes
$$
\begin{aligned}
&A^{(m,0)}_{i_1\dots i_mj_1\dots j_m}g_{i_{m+1}}^{j_1\dots j_m}=m!(2m-1)!!\,|y|^{-1}g_{i_1\dots i_{m+1}}\\
&+\frac{1}{m\!+\!1}\,\sigma(i_1\dots i_m)\sum\limits_{k=1}^m
\sum\limits_{p=0}^{\min(k,m-k)}\!\!\beta(m,k,p)\,
\delta_{i_1i_2}\dots\delta_{i_{2p-1}i_{2p}}\,y_{i_{2p+1}}\dots y_{i_{k+p}}\times\\
&\times
\bigg[\frac{1}{2m\!-\!2k\!+\!1}\,(j_y^p\,A^{(m+1,k)}/g)_{i_{k+p+1}\dots i_{m+1}}
-\frac{k\!-\!p\!+\!1}{k\!+\!1}\,y_{i_{m+1}}\,(j_y^p\,A^{(m+1,k+1)}/g)_{i_{k+p+1}\dots i_m}\\
&\qquad+\frac{m\!-\!k\!-\!p}{k\!+\!1}\,\delta_{i_{m+1}i_{k+p+1}}(j_y^{p+1}\,A^{(m+1,k+1)}/g)_{i_{k+p+2}\dots i_m}
\bigg].
\end{aligned}
$$
Next, we pull the factor $\delta_{i_{2p-1}i_{2p}}\,y_{i_{2p+1}}\dots y_{i_{k+p}}$ inside brackets
\begin{equation}
\begin{aligned}
&A^{(m,0)}_{i_1\dots i_mj_1\dots j_m}g_{i_{m+1}}^{j_1\dots j_m}=m!(2m-1)!!\,|y|^{-1}g_{i_1\dots i_{m+1}}\\
&+\frac{1}{m\!+\!1}\sum\limits_{k=1}^m
\sum\limits_{p=0}^{\min(k,m-k)}\!\!\beta(m,k,p)
\bigg[\frac{1}{2m\!-\!2k\!+\!1}\Big(i_y^{k-p}i^pj_y^p(A^{(m+1,k)}/g)\Big)_{i_1\dots i_{m+1}}\\
&\qquad\qquad\qquad\qquad\qquad\qquad\quad-\frac{k\!-\!p\!+\!1}{k\!+\!1}\,y_{i_{m+1}}\Big(i_y^{k-p}i^pj_y^p(A^{(m+1,k+1)}/g)\Big)_{i_1\dots i_m}\\
&\qquad\qquad\qquad\qquad\qquad\qquad\quad+\frac{m\!-\!k\!-\!p}{k\!+\!1}\Big(
i_{\tilde\delta}\,i_y^{k-p}i^pj_y^{p+1}(A^{(m+1,k+1)}/g)\Big)_{i_1\dots i_m}
\bigg].
\end{aligned}
                           \label{6.17}
\end{equation}

Next, we apply Lemma \ref{L6.1} with $k=0$. More precisely,
we reproduce the formula  \eqref{6.7} from the proof of the lemma for $k=0$
$$
\begin{aligned}
&(A^{(m+1,0)}/g)_{i_1\dots i_{m+1}}
=(m\!+\!1)(2m\!+\!1)\,A^{(m,0)}_{i_1\dots i_m j_1\dots j_m}g_{i_{m+1}}^{j_1\dots j_m}\\
&+(2m\!+\!1)y_{i_{m+1}}\,A^{(m+1,1)}_{i_1\dots i_m j_1\dots j_{m+1}}g^{j_1\dots j_{m+1}}\\
&+m(2m\!+\!1)\,\sigma(i_1\dots i_m)\Big(\delta_{i_{m+1}i_m}\,A^{(m,0)}_{i_1\dots i_{m-1} j_1\dots j_{m+1}}g^{j_1\dots j_{m+1}}\Big).
\end{aligned}
$$
This can be written in the form
$$
\begin{aligned}
(A^{(m+1,0)}&/g)_{i_1\dots i_{m+1}}
=(m\!+\!1)(2m\!+\!1)\,A^{(m,0)}_{i_1\dots i_m j_1\dots j_m}g_{i_{m+1}}^{j_1\dots j_m}\\
&+(2m\!+\!1)y_{i_{m+1}}\,(A^{(m+1,1)}/g)_{i_1\dots i_m}
+m(2m\!+\!1)\Big(i_{\tilde\delta}(A^{(m,0)}/g)\Big)_{i_1\dots i_m}.
\end{aligned}
$$
By \eqref{5.1}, $A^{(m,0)}=-j_yA^{(m+1,1)}$. Substituting this expression into the last line of the previous formula, we obtain
\begin{equation}
\begin{aligned}
&(A^{(m+1,0)}/g)_{i_1\dots i_{m+1}}
=(m\!+\!1)(2m\!+\!1)\,A^{(m,0)}_{i_1\dots i_m j_1\dots j_m}g_{i_{m+1}}^{j_1\dots j_m}\\
&+(2m\!+\!1)y_{i_{m+1}}\,(A^{(m+1,1)}/g)_{i_1\dots i_m}
-m(2m\!+\!1)\Big(i_{\tilde\delta}\,j_y(A^{(m+1,1)}/g)\Big)_{i_1\dots i_m}.
\end{aligned}
                           \label{6.18}
\end{equation}

Now, we replace the first term $A^{(m,0)}_{i_1\dots i_m j_1\dots j_m}g_{i_{m+1}}^{j_1\dots j_m}$ on the right-hand side of
\eqref{6.18} with its expression \eqref{6.17}
\begin{equation}
\begin{aligned}
&\frac{1}{2m\!+\!1}(A^{(m+1,0)}/g)_{i_1\dots i_{m+1}}
=(m\!+\!1)!(2m-1)!!\,|y|^{-1}g_{i_1\dots i_{m+1}}\\
&+\sum\limits_{k=1}^m
\sum\limits_{p=0}^{\min(k,m-k)}\!\!\beta(m,k,p)
\bigg[\frac{1}{2m\!-\!2k\!+\!1}\Big(i_y^{k-p}i^pj_y^p(A^{(m+1,k)}/g)\Big)_{i_1\dots i_{m+1}}\\
&\qquad\qquad\qquad\qquad\qquad\qquad\quad-\frac{k\!-\!p\!+\!1}{k\!+\!1}\,y_{i_{m+1}}\Big(i_y^{k-p}i^pj_y^p(A^{(m+1,k+1)}/g)\Big)_{i_1\dots i_m}\\
&\qquad\qquad\qquad\qquad\qquad\qquad\quad+\frac{m\!-\!k\!-\!p}{k\!+\!1}\Big(
i_{\tilde\delta}\,i_y^{k-p}i^pj_y^{p+1}(A^{(m+1,k+1)}/g)\Big)_{i_1\dots i_m}
\bigg]\\
&+y_{i_{m+1}}\,\,(A^{(m+1,1)}/g)_{i_1\dots i_m}
-m\Big(i_{\tilde\delta}\,j_y(A^{(m+1,1)}/g)\Big)_{i_1\dots i_m}.
\end{aligned}
                           \label{6.19}
\end{equation}

From now on, we again let $i_{m+1}$ be an arbitrary index.
We apply the symmetrization $\sigma(i_1\dots i_{m+1})$ to the equation \eqref{6.19}. The operator $i_{\tilde\delta}$ becomes $i$ after the symmetrization and the result can be written in the coordinate-free form (recall that operators $i_y$ and $i$ commute)
\begin{equation}
\begin{aligned}
&\frac{1}{2m\!+\!1}\,A^{(m+1,0)}/g
=(m\!+\!1)!(2m-1)!!\,|y|^{-1}\,g\\
&+\sum\limits_{k=1}^m
\sum\limits_{p=0}^{\min(k,m-k)}\!\!\beta(m,k,p)
\bigg[\frac{1}{2m\!-\!2k\!+\!1}\,i_y^{k-p}i^pj_y^p(A^{(m+1,k)}/g)\\
&\qquad\qquad-\frac{k\!-\!p\!+\!1}{k\!+\!1}\,i_y^{k-p+1}i^pj_y^p(A^{(m+1,k+1)}/g)
+\frac{m\!-\!k\!-\!p}{k\!+\!1}\,i_y^{k-p}i^{p+1}j_y^{p+1}(A^{(m+1,k+1)}/g)
\bigg]\\
&+i_y\,(A^{(m+1,1)}/g)
-m\,ij_y\,(A^{(m+1,1)}/g).
\end{aligned}
                           \label{6.20}
\end{equation}

Let us write \eqref{6.20} in the form
$$
\begin{aligned}
&\frac{1}{2m\!+\!1}\,A^{(m+1,0)}/g
=(m\!+\!1)!(2m-1)!!\,|y|^{-1}\,g\\
&+\sum\limits_{k=1}^m\sum\limits_{p=0}^{\min(k,m-k)}\!\!
\frac{1}{2m\!-\!2k\!+\!1}\,\beta(m,k,p)\,i_y^{k-p}i^pj_y^p(A^{(m+1,k)}/g)\\
&-\sum\limits_{k'=1}^m\sum\limits_{p=0}^{\min(k',m-k')}\!\!
\frac{k'\!-\!p\!+\!1}{k'\!+\!1}\,\beta(m,k',p)\,i_y^{k'-p+1}i^pj_y^p(A^{(m+1,k'+1)}/g)\\
&+\sum\limits_{k'=1}^m\sum\limits_{p'=0}^{\min(k',m-k')}\!\!
\frac{m\!-\!k'\!-\!p'}{k'\!+\!1}\,\beta(m,k',p')\,i_y^{k'-p'}i^{p'+1}j_y^{p'+1}(A^{(m+1,k'+1)}/g)\\
&+i_y\,(A^{(m+1,1)}/g)
-m\,ij_y\,(A^{(m+1,1)}/g).
\end{aligned}
$$
We change summation variables by $k'=k-1$ in the second sum and by $k'=k-1,p'=p-1$ in the third sum. The formula becomes
\begin{equation}
\begin{aligned}
&\frac{1}{2m\!+\!1}\,A^{(m+1,0)}/g
=(m\!+\!1)!(2m-1)!!\,|y|^{-1}\,g\\
&+\sum\limits_{k=1}^m\sum\limits_{p=0}^{\min(k,m-k)}\!\!
\frac{1}{2m\!-\!2k\!+\!1}\,\beta(m,k,p)\,i_y^{k-p}i^pj_y^p(A^{(m+1,k)}/g)\\
&-\sum\limits_{k=2}^{m+1}\sum\limits_{p=0}^{\min(k-1,m-k+1)}
\frac{k\!-\!p}{k}\,\beta(m,k-1,p)\,i_y^{k-p}i^pj_y^p(A^{(m+1,k)}/g)\\
&+\sum\limits_{k=2}^{m+1}\sum\limits_{p=1}^{\min(k-1,m-k+1)+1}
\frac{m\!-\!k\!-\!p\!+\!2}{k}\,\beta(m,k-1,p-1)\,i_y^{k-p}i^{p}j_y^{p}(A^{(m+1,k)}/g)\\
&+i_y\,(A^{(m+1,1)}/g)
-m\,ij_y\,(A^{(m+1,1)}/g).
\end{aligned}
                           \label{6.21}
\end{equation}

We are going to equate summation limits in three sums on the right-hand side of \eqref{6.21} in order to unite the sums. Then we are going to involve two terms on the last line of \eqref{6.21} into the same sum. This needs some logical and arithmetic analysis.

In the first sum on the right-hand side of \eqref{6.21}, the summation over $k$ can be extended to $1\le k\le m+1$ since $\beta(m,m+1,p)=0$ by the convention \eqref{6.11}. Let us demonstrate that the summation over $p$ can be extended to $0\le p\le\min(k,m-k+1)$. Indeed,
$\min(k,m-k)=\min(k,m-k+1)$ if $k\le m-k$. If $k> m-k$, then there appears one extra term corresponding to $p=m-k+1$ in the first sum. But
$\beta(m,k,m-k+1)=0$ by the convention \eqref{6.11}. Thus, summation limits of the first sum can be replaced with
\begin{equation}
1\le k\le m+1,\quad 0\le p\le\min(k,m-k+1).
                           \label{6.22}
\end{equation}

In the second sum on the right-hand side of \eqref{6.21}, the summation over $k$ can be extended to $1\le k\le m+1$ since $\beta(m,0,p)=0$ by the convention \eqref{6.11}. Let us demonstrate that the summation over $p$ can be extended to $0\le p\le\min(k,m-k+1)$. Indeed,
$\min(k-1,m-k)=\min(k,m-k+1)$ if $k> m-k+1$. If $k\le m-k+1$, then there appears one extra term corresponding to $p=k$ in the second sum. But
$\beta(m,k-1,k)=0$ by the convention \eqref{6.11}. Thus, summation limits of the second sum can be replaced with \eqref{6.22}.

In the third sum on the right-hand side of \eqref{6.21}, the summation over $k$ can be extended to $1\le k\le m+1$ since $\beta(m,0,p-1)=0$ by the convention \eqref{6.11}. The lower summation limit over $p$ can be replaced with zero since $\beta(m,k-1,-1)=0$ by the convention \eqref{6.11}. Let us demonstrate that the upper summation limit over $p$ can be replaced with $\min(k,m-k+1)$. Indeed,
$\min(k-1,m-k+1)+1=\min(k,m-k+1)$ if either $2k<m+2$ or $2k>m+2$. The only critical case is $2k=m+2$ when $\min(k-1,m-k+1)+1=m-k+2$ and $\min(k,m-k+1)=m-k+1$. We are going to loose the term corresponding to $p=m-k+2$ after the replacement. But this term is equal to zero due to the presence of the factor $\frac{m\!-\!k\!-\!p\!+\!2}{k}$.

Thus, summation limits can be replaced with \eqref{6.22} in all sums on the right-hand side of \eqref{6.21}. After the replacement, in \eqref{6.21}, the middle three terms,
$$
\begin{aligned}
&\sum\limits_{k=1}^m\sum\limits_{p=0}^{\min(k,m-k)}\!\!
\frac{1}{2m\!-\!2k\!+\!1}\,\beta(m,k,p)\,i_y^{k-p}i^pj_y^p(A^{(m+1,k)}/g)\\
&-\sum\limits_{k=2}^{m+1}\sum\limits_{p=0}^{\min(k-1,m-k+1)}
\frac{k\!-\!p}{k}\,\beta(m,k-1,p)\,i_y^{k-p}i^pj_y^p(A^{(m+1,k)}/g)\\
&+\sum\limits_{k=2}^{m+1}\sum\limits_{p=1}^{\min(k-1,m-k+1)+1}
\frac{m\!-\!k\!-\!p\!+\!2}{k}\,\beta(m,k-1,p-1)\,i_y^{k-p}i^{p}j_y^{p}(A^{(m+1,k)}/g)
\end{aligned}
$$
can be combined into a single sum,
$$
\sum\limits_{k=1}^{m+1}\sum\limits_{p=0}^{\min(k,m-k+1)}
\tilde\beta(m+1,k,p)\,i_y^{k-p}i^{p}j_y^{p}(A^{(m+1,k)}/g).
$$
Thus, we write \eqref{6.21} in the form
\begin{equation}
\begin{aligned}
&A^{(m+1,0)}/g
=(m\!+\!1)!(2m\!+\!1)!!\,|y|^{-1}\,g\\
&+(2m\!+\!1)\sum\limits_{k=1}^{m+1}\sum\limits_{p=0}^{\min(k,m-k+1)}
\tilde\beta(m+1,k,p)\,i_y^{k-p}i^{p}j_y^{p}(A^{(m+1,k)}/g)\\
&+(2m\!+\!1)\,i_y\,(A^{(m+1,1)}/g)
-m(2m\!+\!1)\,ij_y\,(A^{(m+1,1)}/g),
\end{aligned}
                           \label{6.23}
\end{equation}
where $\tilde\beta(m+1,k,p)$ is defined by \eqref{6.9}.

Finally, we have to include two terms on the last line of \eqref{6.23} into the sum. The term $i_y\,(A^{(m+1,1)}/g)$ corresponds to $(k,p)=(1,0)$ and the term $ij_y\,(A^{(m+1,1)}/g)$ corresponds to $(k,p)=(1,1)$. Therefore we define $\beta(m+1,k,p)$ by \eqref{6.10}.
The formula \eqref{6.23} becomes now
$$
\begin{aligned}
A^{(m+1,0)}/g
&=(m\!+\!1)!(2m\!+\!1)!!\,|y|^{-1}\,g\\
&+\sum\limits_{k=1}^{m+1}\sum\limits_{p=0}^{\min(k,m-k+1)}
\beta(m+1,k,p)\,i_y^{k-p}i^{p}j_y^{p}(A^{(m+1,k)}/g).
\end{aligned}
$$
This coincides with \eqref{6.8} for $m:=m+1$.
\end{proof}

\begin{proof}[Proof of Proposition \ref{P6.1}]
Coefficients $\beta(m,k,p)$ are determined by recurrent formulas \eqref{6.9}--\eqref{6.11}. Nevertheless, the coefficients can be expressed by the explicit formula
\begin{equation}
\beta(m,k,p) = \begin{cases}
(-1)^{k+p+1} \frac{(2m-1)!!}{(2m-2k-1)!!}  2^{-p} \binom{m}{k} \binom{m-k}{p} & \text{for } k > 0, \\
0 & \text{for } k \le 0.
\end{cases}
\label{6.24}
\end{equation}
Indeed, being defined by \eqref{6.24}, $\beta(m,k,p)$ satisfy \eqref{6.11} with the convention \eqref{5.8a}. The formulas \eqref{6.9}--\eqref{6.10} can be equivalently written in the form
\begin{equation}
\beta(m+1,1,0)=(2m+1)\Big(\frac{1}{2m-1}\,\beta(m,1,0)+1\Big),
                           \label{6.25}
\end{equation}
\begin{equation}
\beta(m+1,1,1)=(2m+1)\Big(\frac{1}{2m-1}\,\beta(m,1,1)-m\Big),
                           \label{6.26}
\end{equation}
\begin{equation}
\begin{aligned}
&\beta(m+1,k,p)=(2m+1)\bigg[\frac{1}{2m-2k+1}\,\beta(m,k,p)
-\frac{k-p}{k}\,\beta(m,k-1,p)\\
&+\frac{m-k-p+2}{k}\,\beta(m,k-1,p-1)\bigg]\quad\big((k,p)\neq(1,0),(k,p)\neq(1,1)\big).
\end{aligned}
                           \label{6.27}
\end{equation}
Unlike \eqref{6.9}--\eqref{6.10}, formulas \eqref{6.25}--\eqref{6.27} do not involve $\tilde\beta(m+1,k,p)$.
One can easily prove that equations \eqref{6.25}--\eqref{6.27} are satisfied by values \eqref{6.24}.

We express $g$ from \eqref{6.8}. Substituting the value \eqref{6.24} of $\beta(m,k,p)$ into the expression, we arrive at the formula \eqref{6.0}. This completes the proof of Proposition \ref{P6.1}.
\end{proof}

\section{\texorpdfstring{Commutator formula for $d$ and $i_y$}{Commutator formula for d and i\_y}}

The operators $i_u$ of symmetric multiplication by a tensor $u$ and $j_u$ of contraction with $u$ were defined in Section 2.1. For arbitrary tensors $u$ and $v$,
$$
i_ui_v=i_vi_u=i_{uv}.
$$
These equalities hold in virtue of the fact: the symmetric product of tensors is associative and commutative. Passing to adjoints, we obtain the less obvious statement
$$
j_uj_v=j_vj_u=j_{uv}.
$$

Recall that $i=i_\delta, j=j_\delta$ where $\delta$ is the Kronecker tensor; $d$ is the inner derivative. The operators $d$ and $i$ commute, this fact is proved by a direct calculation in coordinates. But $d$ and $j$ do not commute.

The operator $i_y$ of symmetric multiplication by the vector $y\in{\R}^n$ plays an important role in our arguments.

\begin{proposition} \label{P7.1}
For integers $k\ge0$ and $l\ge0$,
\begin{equation}
d^ki_y^l=\sum\limits_{p=\max(0,k-l)}^k\binom{k}{p}\frac{l!}{(p-k+l)!}\,i_y^{p-k+l}i^{k-p}d^p.
                           \label{7.1}
\end{equation}
\end{proposition}

\begin{proof}
The equality \eqref{7.1} trivially holds in the case of $k=0$ and in the case of $l=0$. In the case $k=l=1$, \eqref{7.1} looks as follows:
\begin{equation}
di_y=i+i_yd.
                           \label{7.2}
\end{equation}
This formula is easily proved by a direct calculation in coordinates. In the case $k=1$ and of arbitrary $l$, \eqref{7.1} looks as follows:
\begin{equation}
di_y^l=li_y^{l-1}i+i_y^ld.
                           \label{7.3}
\end{equation}
This formula is easily proved on the base of \eqref{7.2} by induction in $l$.

Now we prove \eqref{7.1} by induction in $k$. Assuming the validity of \eqref{7.1} for some $k$, we apply the operator $d$ to this equality
$$
d^{k+1}i_y^l=\sum\limits_{p=\max(0,k-l)}^k\binom{k}{p}\frac{l!}{(p-k+l)!}\,(di_y^{p-k+l})i^{k-p}d^p.
$$
By \eqref{7.3},
$di_y^{p-k+l}=(p-k+l)i_y^{p-k+l-1}i+i_y^{p-k+l}d$. Substituting this expression into the previous formula, we obtain
$$
\begin{aligned}
d^{k+1}i_y^l&=\sum\limits_{p=\max(0,k-l)}^k\binom{k}{p}\frac{l!(p-k+l)}{(p-k+l)!}\,i_y^{p-k+l-1}i^{k-p+1}d^p\\
&+\sum\limits_{p'=\max(0,k-l)}^k\binom{k}{p'}\frac{l!}{(l-k+p')!}\,i_y^{p'-k+l}i^{k-p'}d^{p'+1}.
\end{aligned}
$$
Changing the summation variable in the second sum as $p'=p-1$, we write this in the form
\begin{equation}
\begin{aligned}
d^{k+1}i_y^l&=\sum\limits_{p=\max(0,k-l)}^k\binom{k}{p}\frac{l!(p-k+l)}{(p-k+l)!}\,i_y^{p-k+l-1}i^{k-p+1}d^p\\
&+\sum\limits_{p=\max(0,k-l)+1}^{k+1}\binom{k}{p-1}\frac{l!}{(l-k+p-1)!}\,i_y^{p-k+l-1}i^{k-p+1}d^p.
\end{aligned}
                           \label{7.4}
\end{equation}

We assume the binomial coefficients to be defined for all integers with the convention \eqref{5.8a}. With the convention, the summation limits of the first sum on \eqref{7.4} can be extended to
$$
\max(0,k-l)\le p\le k+1.
$$
Let us demonstrate that the summation limits of the first sum on \eqref{7.4} can be changed to
$$
\max(0,k-l+1)\le p\le k+1.
$$
Indeed, $\max(0,k-l)=\max(0,k-l+1)$ if $k-l<0$. If $k-l\ge0$, then $\max(0,k-l)=k-l$ and $\max(0,k-l+1)=k-l+1$. In the latter case, the term corresponding to $p=k-l$ is equal to zero because of the presence of the factor $p-k+l$.

Similar arguments show that the summation limits of the second sum on \eqref{7.4} can be changed to
$\max(0,k-l+1)\le p\le k+1$. The formula \eqref{7.4} takes the form
$$
\begin{aligned}
d^{k+1}i_y^l&=\sum\limits_{p=\max(0,k-l+1)}^{k+1}\binom{k}{p}\frac{l!}{(p-k+l-1)!}\,i_y^{p-k+l-1}i^{k-p+1}d^p\\
&+\sum\limits_{p=\max(0,k-l+1)}^{k+1}\binom{k}{p-1}\frac{l!}{(l-k+p-1)!}\,i_y^{p-k+l-1}i^{k-p+1}d^p.
\end{aligned}
$$
We can now unite two sums
$$
d^{k+1}i_y^l=\sum\limits_{p=\max(0,k-l+1)}^{k+1}\bigg[\binom{k}{p}+\binom{k}{p-1}\bigg]\frac{l!}{(p-k+l-1)!}\,i_y^{p-k+l-1}i^{k-p+1}d^p.
$$
By the Pascal triangle equality, $\binom{k}{p}+\binom{k}{p-1}=\binom{k+1}{p}$ and our formula becomes
$$
d^{k+1}i_y^l=\sum\limits_{p=\max(0,k+1-l)}^{k+1}\binom{k+1}{p}\frac{l!}{(p-(k+1)+l)!}\,i_y^{p-(k+1)+l}i^{k+1-p}d^p.
$$
This finishes the induction step.
\end{proof}
Recall that $i_y^*=j_y$ and $d^*=-\mbox{div}$. Passing to adjoints in \eqref{7.1}, we obtain the commutator formula for $\mbox{div}$ and $j_y$
$$
j_y^l\,\mbox{div}^k=(-1)^k\sum\limits_{p=\max(0,k-l)}^k(-1)^p\binom{k}{p}\frac{l!}{(p-k+l)!}\,\mbox{div}^p\,j^{k-p}j_y^{p-k+l}.
$$
As we have mentioned in Section 2, the operators $j$ and $\mbox{div}$ commute. Therefore the last formula can be written in the form
\begin{equation}
j_y^l\,\mbox{div}^k=(-1)^k\sum\limits_{p=\max(0,k-l)}^k(-1)^p\binom{k}{p}\frac{l!}{(p-k+l)!}\,j^{k-p}\,\mbox{div}^p\,j_y^{p-k+l}.
                           \label{7.5}
\end{equation}

\section{Proof of Theorem \ref{Th3.1}}

The formula \eqref{6.0} expresses the solution to the system \eqref{3.8} with arbitrary right-hand sides $H^{(m,k)}$ (if the system is solvable). Let us now consider the case of $H^{(m,k)}$ defined by \eqref{3.9}. In this case, the formula \eqref{6.0} can be essentially simplified. Indeed, applying the operator $j_y^p$ to the formula \eqref{3.9}, we obtain
\begin{equation}
j_y^pH^{(m,k)}=\frac{(2m-2k-1)!!(n+2m-2k-3)!!}{(2m-1)!!(n+2m-3)!!}
\sum\limits_{q=0}^k(-1)^q\binom{k}{q}j_y^p\,\mbox{\rm div}^{k-q}\,F^{(m,q)}.
                           \label{8.1}
\end{equation}
We consider \eqref{8.1} for $0\le p\le\min(k,m-k)$ since such $p$ participate in \eqref{6.0}.

By \eqref{7.5},
\begin{align}
    \begin{split}
     &  j_y^p\,\mbox{\rm div}^{k-q}\,F^{(m,q)}\\&\qquad=(-1)^{k-q}\!\!\!\!\!\sum\limits_{r=\max(0,k-p-q)}^{k-q}\!\!\!\!\!
(-1)^r\binom{k-q}{r}\frac{p!}{(r\!-\!k\!+\!p\!+\!q)!}\,j^{k-q-r}\,\mbox{\rm div}^r\,j_y^{r-k+p+q}F^{(m,q)}.
                           \label{8.2}
    \end{split}
\end{align}
We recall that $j_yF^{(m,p)}=0$, see \eqref{4.8}. Therefore there is at most one non-zero summand corresponding to $r=k-p-q$ on the right-hand side of \eqref{8.2}. In this way \eqref{8.2} simplifies to the formula
\begin{equation}
j_y^p\,\mbox{\rm div}^{k-q}\,F^{(m,q)}=
\left\{\begin{array}{ll}
(-1)^p\binom{k-q}{p}p!\,j^p\,\mbox{\rm div}^{k-p-q}\,F^{(m,q)}&\mbox{if}\ p+q\le k,\\
0&\mbox{if}\ p+q>k.\end{array}\right.
                           \label{8.3}
\end{equation}
Substituting the expression \eqref{8.3} into \eqref{8.1}, we obtain
\begin{equation}
\begin{aligned}
j_y^pH^{(m,k)}=\frac{(-1)^p(2m\!-\!2k\!-\!1)!!(n\!+\!2m\!-\!2k\!-\!3)!!}{(2m\!-\!1)!!(n\!+\!2m\!-\!3)!!}\,\frac{k!}{(k\!-\!p)!}
&\,j^p\!\!\!\sum\limits_{q=0}^{\min(k,k-p)}\!\!(-1)^q\binom{k-p}{q}\times\\
&\times \mbox{\rm div}^{k-p-q}\,F^{(m,q)}.
\end{aligned}
                           \label{8.4}
\end{equation}

Next, we replace the last factor $j_y^pH^{(m,k)}$ on the right-hand side of \eqref{6.0} with its value \eqref{8.4}
$$
\begin{aligned}
g(y)=\frac{|y|}{m!(2m\!-\!1)!!(n\!+\!2m\!-\!3)!!}&\sum\limits_{k=0}^m(-1)^k \binom{m}{k}(n\!+\!2m\!-\!2k\!-\!3)!!
\sum\limits_{p=0}^{\min(k,m-k)}\!\!2^{-p}\binom{m - k}{p}\times\\
&\times \frac{k!}{(k\!-\!p)!}\sum\limits_{q=0}^{\min(k,k-p)}\!\!(-1)^q\binom{k-p}{q}
i^pi_y^{k-p}j^p\,\mbox{\rm div}^{k-p-q}\,F^{(m,q)}.
\end{aligned}
$$
Since the operators $i$ and $i_y$ commute, this can be written in the form
\begin{align}\label{eq:p:100}
    \begin{split}
        (2m\!-\!1)!!(n\!+\!2m\!-\!3)!!\,g(y)&=|y|\,\sum\limits_{k=0}^m\sum\limits_{p=0}^{\min(k,m-k)}\sum\limits_{q=0}^{\min(k,k-p)}(-1)^{k+q}\times\\
&\times \frac{(n\!+\!2m\!-\!2k\!-\!3)!!}{2^p\,p!\,q!(m\!-\!k\!-\!p)!(k\!-\!p\!-\!q)!}\,
i_y^{k-p}i^pj^p\,\mbox{\rm div}^{k-p-q}\,F^{(m,q)}.
    \end{split}
\end{align}
Changing the order of summations, we write formula \eqref{eq:p:100} as follows:
$$
\begin{aligned}
&(2m\!-\!1)!!(n\!+\!2m\!-\!3)!!\,g(y)=|y|\,\sum\limits_{q=0}^m\frac{(-1)^q}{q!}\times\\
&\times\!\bigg(\sum\limits_{k=q}^m(-1)^k(n\!+\!2m\!-\!2k\!-\!3)!! \!\!\!\!\!\!\sum\limits_{p=0}^{\min(k,m-k,k-q)}\!\!\!\!\!\!\!\!\!
\frac{1}{2^p\,p!(m\!-\!k\!-\!p)!(k\!-\!p\!-\!q)!}\,
i_y^{k-p}i^pj^p\,\mbox{\rm div}^{k-p-q}\bigg)F^{(m,q)}.
\end{aligned}
$$
Finally, changing notation of summation variables, we write the formula in the final form
\begin{equation}
\begin{aligned}
&(2m\!-\!1)!!(n\!+\!2m\!-\!3)!!\,g(y)=|y|\,\sum\limits_{k=0}^m\frac{(-1)^k}{k!}\times\\
&\times\!\bigg(\sum\limits_{p=k}^m(-1)^p(n\!+\!2m\!-\!2p\!-\!3)!! \!\!\!\!\!\!\sum\limits_{q=0}^{\min(p,m-p,p-k)}\!\!\!\!\!\!\!\!\!
\frac{1}{2^q\,q!(m\!-\!p\!-\!q)!(p\!-\!k\!-\!q)!}\,
i_y^{p-q}i^qj^q\,\mbox{\rm div}^{p-k-q}\bigg)F^{(m,k)}.
\end{aligned}
                           \label{8.5}
\end{equation}

Recall that we denoted $\widehat f$ by $g$ and $F^{(m,k)}$ was defined by \eqref{3.6} in Section 3. Replacing $g$ with $\widehat f$ and replacing $F^{(m,k)}$ with its value \eqref{3.6}, we write the formula \eqref{8.5} as follows:
\begin{equation}
\begin{aligned}
\widehat f(y)&=\frac{2^{m-2}\Gamma\big(\frac{2m+n-1}{2}\big)}{\pi^{(n+1)/2}(n+2m-3)!!}\,|y|\,\sum\limits_{k=0}^m\frac{(-1)^k}{(k!)^2}
\bigg(\sum\limits_{p=k}^m(-1)^p(n\!+\!2m\!-\!2p\!-\!3)!! \\
&\times\sum\limits_{q=0}^{\min(p,m-p,p-k)}
\frac{1}{2^q\,q!(m\!-\!p\!-\!q)!(p\!-\!k\!-\!q)!}\,
i_y^{p-q}i^qj^q\,\mbox{\rm div}^{p-k-q}\,j_y^k\bigg)\widehat{N_m^kf}.
\end{aligned}
                           \label{8.6}
\end{equation}
Introducing the differential operator
\begin{align}
    \begin{split}
        \widehat D^k_{m,n}&=c^k_{m,n}\sum\limits_{p=k}^m(-1)^p(n\!+\!2m\!-\!2p\!-\!3)!! \!\!\!\!\\&\qquad\times\sum\limits_{q=0}^{\min(p,m-p,p-k)}\!\!\!\!\!\!\!\!\!
\frac{1}{2^q\,q!(m\!-\!p\!-\!q)!(p\!-\!k\!-\!q)!}\,
i_y^{p-q}i^qj^q\,\mbox{\rm div}^{p-k-q}\,j_y^k
                           \label{8.7}
    \end{split}
\end{align}
with the coefficient $c^k_{m,n}$ defined by \eqref{3.3}, we write \eqref{8.6} in the form
\begin{equation}
\widehat f(y)=|y|\,\sum\limits_{k=0}^m\widehat D^k_{m,n}\widehat{N_m^kf}(y).
                           \label{8.8}
\end{equation}

It remains to apply the inverse Fourier transform ${\mathcal F}^{-1}$ to the formula \eqref{8.8}. Since
${\mathcal F}^{-1}|y|=(-\Delta)^{1/2}{\mathcal F}^{-1}$, we obtain
\begin{equation}
f=(-\Delta)^{1/2}\sum\limits_{k=0}^m D^k_{m,n}(N_m^kf),
                           \label{8.9}
\end{equation}
where $D^k_{m,n}={\mathcal F}^{-1}\,\widehat D^k_{m,n}\,{\mathcal F}$. Together with \eqref{8.7}, the latter equality gives
\begin{equation}
\begin{aligned}
D^k_{m,n}=c^k_{m,n}\sum\limits_{p=k}^m(-1)^p(n\!+\!2m\!-\!2p\!-\!3)!! &\sum\limits_{q=0}^{\min(p,m-p,p-k)}
\frac{1}{2^q\,q!(m\!-\!p\!-\!q)!(p\!-\!k\!-\!q)!}\times\\
&\times {\mathcal F}^{-1}\,i_y^{p-q}i^qj^q\,\mbox{\rm div}^{p-k-q}\,j_y^k\,{\mathcal F}.
\end{aligned}
                           \label{8.10}
\end{equation}

Using the commutator formulas ($\textsl{i}$ is the imaginary unit)
$$
\begin{aligned}
&{\mathcal F}^{-1}\,i_y=-\textsl{i}\,d\,{\mathcal F}^{-1},\quad
{\mathcal F}^{-1}\,j_y=-\textsl{i}\,\mbox{div}\,{\mathcal F}^{-1},\quad
{\mathcal F}^{-1}\,\mbox{div}=-\textsl{i}\,j_x\,{\mathcal F}^{-1},\\
&{\mathcal F}^{-1}i=i{\mathcal F}^{-1},\quad {\mathcal F}^{-1}j=j{\mathcal F}^{-1},
\end{aligned}
$$
we obtain
$$
{\mathcal F}^{-1}\,i_y^{p-q}i^qj^q\,\mbox{\rm div}^{p-k-q}\,j_y^k\,{\mathcal F}
=(-1)^{p+q}\,d^{p-q}\,i^q\,j^q\,j_x^{p-k-q}\,\mbox{\rm div}^k.
$$
Inserting this expression into \eqref{8.10}, we get
\begin{align}
    \begin{split}
    &D^k_{m,n}=c^k_{m,n}\sum\limits_{p=k}^m(n\!+\!2m\!-\!2p\!-\!3)!! \!\!\!\!\\&\qquad\qquad\qquad\times\sum\limits_{q=0}^{\min(p,m-p,p-k)}\!\!\!\!
\frac{(-1)^q}{2^q\,q!(m\!-\!p\!-\!q)!(p\!-\!k\!-\!q)!}
\,d^{p-q}\,i^q\,j^q\,j_x^{p-k-q}\,\mbox{\rm div}^k.
                           \label{8.12}
    \end{split}
\end{align}
The formulas \eqref{8.9} and \eqref{8.12} coincide with \eqref{3.1} and \eqref{3.2} respectively. This completes the proof of Theorem \ref{Th3.1}.

\section*{Acknowledgements}
SRJ and VPK would like to thank the Isaac Newton Institute for Mathematical Sciences, Cambridge, UK, for support and hospitality during \emph{Rich and Nonlinear Tomography - a multidisciplinary approach} in 2023 where part of this work was done (supported by EPSRC Grant Number EP/R014604/1). Additionally, VPK acknowledges the support of the Department of Atomic Energy,  Government of India, under
Project no.  12-R\&D-TFR-5.01-0520. SRJ acknowledges the Prime Minister's Research Fellowship (PMRF) from the Government of India for his PhD work. MK was supported by MATRICS grant (MTR/2019/001349) of SERB.
The work of VAS was performed according to the Russian Government research assignment for IM~SB~RAS, project FWNF-2022-0006. We thank the anonymous referees for their many helpful comments that improved the final manuscript.

\bibliography{math}
\bibliographystyle{alpha}


\end{document}